\documentclass[12pt]{amsart}
\usepackage{amscd}
\usepackage{amsmath}
\usepackage{amssymb}
\usepackage{latexsym}
\usepackage{amsthm}

\newtheorem{theorem}{Theorem}[section]
\newtheorem{lemma}[theorem]{Lemma}

\theoremstyle{theorem}
\newtheorem{definition}[theorem]{Definition}

\newtheorem{proposition}[theorem]{Proposition}
\newtheorem{corollary}[theorem]{Corollary}

\newtheorem{assumption}[theorem]{Assumption}

\theoremstyle{remark}
\newtheorem{remark}[theorem]{Remark}

\numberwithin{equation}{section}

\def\C{\mathbb{C}}
\def\R{\mathbb{R}}
\def\Z{\mathbb{Z}}
\def\F{\mathbb{F}}

\def\W{\mathcal{W}_\chi}

\def\sumw{\sum_{w\in W}}
\def\a{\alpha}

\def\la{\lambda}
\def\La{\Lambda}
\def\ga{\gamma}
\def\Ga{\Gamma}

\def\T{\widetilde{S}}

\def\G{\widetilde{G}}
\def\M{\widetilde{M}}
\def\K{\widetilde{K}}

\def\P{\widetilde{P}}
\def\s{\widetilde{S}}
\def\J{\widetilde{J}}

\def\sgn{\operatorname{sgn}} % notation for the sign function, but I'll want to check Step's template.
\def\Hom{\operatorname{Hom}}

\def\End{\operatorname{End}}

\def\Res{\operatorname{Res}}

\def\Gal{\operatorname{Gal}}\def\Spec{\operatorname{Spec}}

\newcommand{\conj}[1]{\overline{#1}}

\newcommand{\map}[2]{: #1\to #2}

\def\p{\varpi} % the uniformiser

\def\sss{{\bf s}}
\def\fbar{\overline{F}}

\newcommand{\inv}{^{-1}}

\def\Gm{\mathbb{G}_m}

\def\lqt{\backslash} % left quotient

\numberwithin{equation}{section}

\newcommand{\arxiv}[1]{\href{http://arxiv.org/abs/#1}{\tt arXiv:\nolinkurl{#1}}}
\newcommand{\arXiv}[1]{\href{http://arxiv.org/abs/#1}{\tt arXiv:\nolinkurl{#1}}}

\usepackage{hyperref}

\title[]{The metaplectic Casselman-Shalika formula}
\address{}\email{mcnamara@maths.usyd.edu.au}
\author{Peter J McNamara}
\date{\today}

\begin{document}
\maketitle

\begin{abstract}
This paper studies spherical Whittaker functions for central extensions of reductive groups over local fields. We follow the development of Chinta and Offen to produce a metaplectic Casselman-Shalika formula for tame covers of all unramified groups.
\end{abstract}

\section{Introduction}

Suppose that $G$ is an unramified reductive group over a non-archimedean local field $F$. By definition, this means that $G$ is the generic fibre of a smooth reductive group scheme over the valuation ring $O_F$, or equivalently that $G$ is quasi-split and splits over an unramified extension of $F$. The Casselman-Shalika formula is an explicit formula for the spherical Whittaker function that is associated to an unramified principal series representation of $G(F)$.

In this paper, we replace $G(F)$ by a central extension of it by a finite cyclic group, and develop a Casselman-Shalika formula for this so-called metaplectic group. Our main result is the union of Theorems \ref{formal}, \ref{sl2tau} and \ref{su3tau}. Theorem \ref{formal} is the metaplectic Casselman-Shalika formula, computing the metaplectic Whittaker function in terms of a certain Weyl group action. Theorems \ref{sl2tau} and \ref{su3tau} describe how a simple reflection acts in this Weyl group action. Along the way, in Section \ref{gkformula}, we prove a metaplectic Gindikin-Karpelevic formula.

Our approach is to follow the technique of Chinta and Offen \cite{chintaoffen} who have shown how to generalise the approach of Casselman and Shalika \cite{cs80} to provide a formula for a Whittaker function on the metaplectic cover of $GL_r$. The purpose of this paper is to show how their technique generalises to the more general case of covers of unramified groups.

This paper naturally splits into two parts. In the first part of this paper, we work in the generality of considering any finite cyclic cover of any reductive $p$-adic group modulo Assumptions \ref{assumption1} and \ref{abelian}.
These assumptions are valid in the most important case when $G$ is simply connected and unramified. For more general $G$ these are more subtle issues and there is some discussion of this in the body of the paper.
This part of the paper culminates in the aforementioned Theorem \ref{formal}, and closely follows the approach of Chinta and Offen. The second part begins with Section \ref{parttwo} and develops the necessary extra results to enable one to compute this Whittaker function when the underlying reductive group is unramified.

To conclude, we compare our computation of the metaplectic Whittaker function with the objects appearing in the local part of a Weyl group multiple Dirichlet series constructed by Chinta and Gunnells \cite{cg}.

Throughout, we work with an assumption that $q\equiv 1\pmod{2n}$ (where $q$ is the cardinality of the residue field and $n$ is the degree of the central extension)
which we feel is worth remarking on. In writing this paper, we have strived to work with the greatest possible range of groups $G$. In this generality, we are not sure how to proceed at times under what might be considered the more natural assumption that $q\equiv 1\pmod{n}$.
As an indication of the simplification this assumption entails, one may look at the setup of Weissman \cite{weissmanl}, where the assumption $q\equiv 1\pmod{2n}$ corresponds to trivialising his second twist $\chi$. However we expect that working only with $q\equiv 1\pmod{n}$ will ultimately only introduce signs into various formulae we encounter.

We have chosen not to discuss the spherical function as we do not believe we can offer anything novel to say about it. The techniques of Casselman \cite{casselmanshperical} are sufficient to compute the metaplectic spherical function, for example see \cite[\S 10]{chintaoffen} for a treatement of the case when $G=GL_n$.

% We have chosen not to attempt this here since we feel that the extra technical complication will detract from the main narrative of this paper.

The author would like to thank D. Bump, G. Chinta, G. Hainke, O. Offen, A. Puskas and some anonymous help for useful discussions.

%%%%%%%%%%%%%%%%%%%%%%%%%%%%%%%%%%%%%%%%%%%%%%%%%%%%%%%%%%%%%%%%%%%%%%%
\section{The metaplectic group}\label{metaplectic}%%%%%%%%%%%%%%%%%%%%%
%%%%%%%%%%%%%%%%%%%%%%%%%%%%%%%%%%%%%%%%%%%%%%%%%%%%%%%%%%%%%%%%%%%%%%%

Fix a positive integer $n$. Let $F$ be a non-archimedean local field with valuation ring $O_F$, uniformiser $\p$ and residue field of order $q$. We assume $q\equiv 1\pmod{2n}$. Let $\mu_n$ denote the group of $n$-th roots of unity,
this is a cyclic group of order $n$.
Fix, once and for all, an embedding $\epsilon\map{\mu_n}{\C^\times}$.
% We anticipate that the ideas of this paper are generalisable to the case where $q$ is only assumed to be congruent to 1 modulo $n$. We shall not attempt this here, since we anticipate the extra technical work will detract from the themes of this paper.

Let $G$ be a connected reductive algebraic group over $F$. Let $S$ be a maximal split torus of $G$ and let $T$ be a maximal torus of $G$ containing $S$. We use the following theorem to find a split reductive subgroup $G'$ of $G$ that will be of use to us.

\begin{theorem}\cite[Th\'eor\`eme 7.2]{boreltits}\label{reducetosplit}
Let $\Phi=\Phi(S,G)$ be the root system of $G$ relative to $S$, and let $\Phi'$ be the set of roots $a$ for which $2a$ is not a root. Let $\Delta$ be a choice of simple roots in $\Phi'$. For each $a\in \Delta$, let $E_a$ be a one-dimensional subgroup of the root subgroup of $a$, and let $V$ be the group generated by the $E_a$. Then $G$ possesses a unique split reductive subgroup $G'$ containing $S$ and $V$. The torus $S$ is a maximal torus of $G'$ and the root system is $\Phi'=\Phi(G',S)$. In particular, the Weyl groups of $G$ and $G'$ are isomorphic.
\end{theorem}

Let $\fbar$ be a separable closure of $F$. We write $W$ for the Weyl group of $G$ and $W_{\fbar}$ for the Weyl group of $G_{\fbar}$, the base change of $G$ to $\fbar$. Consider the geometric cocharacter lattice $X_*(T_{\fbar})=\Hom_{\fbar}(\Gm,T_{\fbar})$. This is equipped with actions of both $W_{\fbar}$ and the Galois group $\Gal(\fbar/F)$.

% Let us assume that $Q$ takes non-zero values on all simple s. This is a harmless a as when $Q$ is zero, the corresponding central extension discussed in the next paragraph is trivial.

Brylinski and Deligne \cite[Theorem 7.2]{bd} provide a classification of central extensions of $G$ by $K_2$ as sheaves on the big Zariski site over $\Spec(F)$. Since $H^1(F,K_2)=0$, at the level of $F$-points, this gives a central extension of $G$ by $K_2(F)$. We push forward such a central extension by the Hilbert symbol $K_2(F)\rightarrow \mu_n$ to obtain a central extension $\G$ of $G$ by $\mu_n$. Explicitly, such a central extension is a short exact sequence of topological groups
$$
1\rightarrow \mu_n \rightarrow \G \rightarrow G\rightarrow 1
$$ with $\mu_n$ lying in the centre of $\G$. We write $p\map{\G}{G}$ for the projection map. On occasion, we shall find it necessary to express $\G$ in terms of a 2-cocycle on $G$. When this is the case, we typically denote the section $G\mapsto \G$ by $\sss$ and the chosen 2-cocycle by $\sigma$. Note that $\sss$ is not a homomorphism, the multiplication in $\G$ is given by $\sss(g_1g_2)=\sss(g_1)\sss(g_2)\sigma(g_1,g_2)$.

All representations of $\G$ which we consider will be representations where the central $\mu_n$ acts by scalars by the embedding $\epsilon$. Such representations are called \emph{genuine}. We often invoke the convention of omitting $\epsilon$ from the notation.

The major piece of data classifying central extensions of $G$ by $K_2$ is a $\Gal(\fbar/F)$ and $W_{\fbar}$-invariant quadratic form $Q\map{\Hom_{\fbar}(\Gm,T_{\fbar})}{\Z}$. Let us fix once and for all such a quadratic form $Q$.

When the group $G$ does not split over $F$, the metaplectic cover can be constructed via a descent process starting from the split case. We will now give a brief description of the most relevant part of this construction when the group $G$ is unramified. This is the case which is of greatest interest to us. This process is described with more detail in \cite[\S 12.11]{bd}. Let $E$ be a degree $d$ unramified extension of $F$ over which $G$ splits.

We think about the construction of $\G$ as a three step process. First one constructs the central extension of $G(F)$ by $K_2(F)$. Then one pushes forward by the tame symbol $K_2(F)\rightarrow k^\times$ to arrive at an extension of $G$ by $k^\times$. One finally pushes forward this extension via the operation of raising to the $(q-1)/n$-th power to obtain the metaplectic group $\G$.

The Brylinski-Deligne construction shows in this case that $\G$ can be realised as a subgroup of a group $\widetilde{G(E)}$ which is a central extension of $G(E)$ by the group of $n\frac{q^d-1}{q-1}$-th roots of unity.

For any subgroup $H$ of $G$, we denote by $\widetilde{H}$ the inverse image of $H$ in $\G$.

Let $B(x,y)=Q(x+y)-Q(x)-Q(y)$ be the bilinear form on $X_*(T_{\fbar})$ associated to $Q$. The commutator map $[\cdot,\cdot]\map{\widetilde{T}\times \widetilde{T}}{\mu_n}$ factors through $p\times p\map{\widetilde{T}\times \widetilde{T}}{T\times T}$ to give a well-defined map from $T\times T$ to $\mu_n$, also denoted $[\cdot,\cdot]$. It takes the following form \cite[Corollary 3.14]{bd}:
\begin{equation}\label{commutatorformula}
[x^\la,y^\mu]=(x,y)^{B(\la,\mu)},
\end{equation}
where $(\cdot,\cdot)\map{F^\times\times F^\times}{\mu_n}$ is the Hilbert Symbol.

There is an explicit formula for the Hilbert symbol in the tame case which we are in, namely
\begin{equation}\label{hilbert}
(x,y)=\left( \overline{(-1)^{v(x)v(y)} \frac{x^{v(y)}}{y^{v(x)})}} \right)^{\frac{q-1}{n}}
\end{equation}
where the bar means to take the reduction modulo $\varpi$ of an element of $O_F$. A choice of an embedding $\mathbb{F}_q^\times \to \C^\times$ is needed to identify the image of $(\cdot,\cdot)$ as lying in $\mu_n$.

Let us now restrict our attention to the central extension $\G'$ of $G'$. There is a natural inclusion of cocharacter groups $X_*(S)\subset X_*(T_{\fbar})$. The restriction of $Q$ to $X_*(S)$ is a $W$-invariant quadratic form and $\G'$ is the central extension of $G'$ associated to $Q$ by the Brylinski-Deligne theory.

Since $G'$ is split, we can, and will make use of the theory developed in \cite{mcn} where only covers of split groups were considered. 

We may identify the Bruhat-Tits building of $G'$ with a subset of the Bruhat-Tits building of $G$. Pick a hyperspecial point in the apartment corresponding to $S$, and let ${\bf G}$ be the corresponding group scheme over $O_F$ with special fibre $G$ via Bruhat-Tits theory. We let $K={\bf G}(O_F)$, this is a maximal compact subgroup of $G$.

We will need to define a lift for any element of $W$ into $G$. To achieve this, it suffices to work inside the group $G'$. By a theorem of Tits \cite{tits}, for each simple reflection $s_\a\in W$, we can choose $w_\a\in {\bf G'}(O_F)$ such that the collection of $w_\a$'s so obtained satisfy the braid relations. Furthermore, if we consider the natural projection from the group generated by the $w_\a$ to $W$, the kernel is an elementary abelian 2-group contained in $S$. Thus for any $w\in W$, we define a lift by writing $w=s_{\a_1} \cdots s_{\a_N}$ as a reduced product of simple reflections, and letting the lift be the product $w_{\a_1}\cdots w_{\a_N}$. Let us denote by $W_0$ the subgroup of $G$ generated by the $w_\a$.

Let $M$ be the centraliser of $S$ in $G$. This is a minimal Levi subgroup of $G$.

\begin{lemma}\label{kerproj}
The kernel of the projection from $W_0$ to $W$ lies in the centre of $\M$.
\end{lemma}
\begin{proof}

For a general element $m\in \M$, $p(m)$ is semisimple, so lies in a maximal torus $T'$ of $G$. The central extension $\G$ is obtainable via a descent process from a central extension of $G(E)$ where $E$ is some Galois extension of $F$. To check whether $m$ and $z$ commute, it suffices to perform this computation in the central extension of $G(E)$. We choose $E$ so that the torus $T'$ splits. This reduces the problem to the split case.

Now suppose $G$ is split and $z$ is in the aforementioned kernel. Note that $z^2=1$. If $m\in \T$, then by the explicit formula (\ref{commutatorformula}) for the commutator, we compute $[m,z]=1$, since $(q-1)/n$ is even and hence the Hilbert symbol $(-1,x)$ is always 1.
\end{proof}
%Since $p(z)\in S$ which is central in $M$, there is a group homomorphism from $M$ to $\mu_n$ given by $m\mapsto [\tilde{m},z]$. 
%Choose $m\in \M$. The group $M$ contains no unipotent elements, so there exists a maximal torus $T'\subset M$ containing $p(m)$.
%The central extension $\G$ is obtainable from a central extension of 

%Since $\M$ is the union of the conjugates of $\T$, it suffices to show that this homomorphism is trivial on $\T$. But to show this, it suffices to pass to a field extension where $\G$ is split. Using our formulae for the commutator in the split case, we see that $z$ is central in $\T$ as we are assuming that $(q-1)/n$ is even, and we are done.

At all stages, unless explicitly mentioned otherwise, we choose normalisations of Haar measures such that the intersection with the maximal compact subgroup $K$ has volume 1.

%There are cal isomorphisms $\s/H\cong Y/\La$ and $H/(H\cap \mu_n K)\cong \La$.

%We now discuss a couple of splitting properties. 
A subgroup $J$ of $G$ is said to be split in the extension if there is a section $J\mapsto \J$ that is a group homomorphism. When this occurs, we also denote the image of $J$ in $\G$ by $J$.
\begin{theorem}\cite[Proposition 4.1]{mcn}\label{unipotentsplitting}
Any unipotent subgroup of $G$ has a unique splitting.
\end{theorem}
\begin{remark}
 In the reference quoted the splitting is said to be canonical but an examination of the proof shows it is unique.
\end{remark}

We will make the following assumption, necessary to make sense of the notion of a spherical Whittaker function. 
\begin{assumption}\label{assumption1}
The subgroup $K$ has a splitting.
\end{assumption}
When $G$ is not simply connected, the splitting or otherwise of a Brylinski-Deligne extension over $K$ is subtle. For example \cite[\S 4.6]{gg} gives an example of a Brylinski-Deligne extension of $PGL_2$ by $\mu_2$ which does not split over $PGL_2(O_F)$. In the same paper a criterion for splitting the subgroup $K$ is given \cite[Corollary 4.2]{gg} for all split $G$. The family of central extensions $\G$ which satisfy the condition of Assumption \ref{assumption1} via this criterion contains a large family of interesting central extensions.

We now explain why, when $G$ is simply connected and unramified, that making this assumption on the splitting of $K$ does not entail any loss of generality.

If $G$ is split, then Assumption \ref{assumption1} is true by \cite[Corollary 4.2]{gg}. Now consider any unramified simply connected $G$. It is possible to pushforward the central extension $\G$ by the inclusion $\mu_n\hookrightarrow S^1$ and get a central extension of $G$ by the group of complex numbers of norm 1. Doing so does not in any way change the representation theory of $\G$. 
The group $\K$ which is now a central extension of $K$ by $S^1$ is a subgroup of a split group $\widetilde{L}$ which is the inverse image in the central extension of a maximal compact subgroup after making an unramified base change where $G$ splits. Therefore the group $\K$ splits under the central extension by $S^1$, which is sufficient for our representation theoretic purposes.

The splitting of $K$ is not unique in general. We will choose once and for all such a splitting. Since the group $W_0$ lies in $K$, this allows us to define a lift of any element $w\in W$ to an element of $\G$, which we will by an abuse of notation also call $w$.

We will need to choose a lift of $X_*(S)$ into $\G$. To achieve this, 
we first use the uniformiser $\p$ to determine a homomorphism from $X_*(S)$ to $G$, namely $\la\mapsto \p^\la$. Then we have a central extension
\begin{equation}\label{extxs}
1\to \mu_n \to\widetilde{X_*(S)} \to X_*(S)\to 1
\end{equation}
which is abelian by the commutator formula (\ref{commutatorformula}). At this point we are using the assumption that $2n$ divides $q-1$ which has the consequence that $(\p,\p)=1$.

Since $X_*(S)$ is a free abelian group, the short exact sequence (\ref{extxs}) of abelian groups splits. We choose a splitting and also denote it by $\la\mapsto \p^\la\in\G$.

%it suffices to work with the split subgroup $G'$ of $G$. Thus we may use the cocycle present in \cite{mcn}. In this manner, we define our lift by $\la\mapsto \sss(\p^\la)$ and commonly denote this image simply by $\p^\la$. Since $2n$ divides $q-1$, the Hilbert symbol $(\p,\p)$ is trivial and this map $X_*(S)\to\G$ is a homomorphism.

Now that we have a splitting of $K$, we define a subgroup $H$ of $\M$ as follows:
\[
 H=\{h\in \M \mid [h,\eta]\in K\ \text{for all}\  \eta\in \M\cap K\}.
\]
If $G$ happens to be quasisplit, then this is simply the centraliser in $\M$ of $M\cap K=\M\cap K$.

\begin{lemma}\label{inh}
Let $s\in S$. Then $\sss(s^n)\in H$.
\end{lemma}

\begin{proof}
Fix $m\in \M$. As $S$ is central in $M$, the map $\psi_m\map{S}{\mu_n}$ defined by $\psi_m(s)=[m,s]$ is a group homomorphism, hence trivial on all $n$-th powers.
\end{proof}

For any coroot $\a$, define $n_\a=n/\gcd(n,Q(\a))$.

\begin{lemma}
Suppose $G$ is unramified and let $\a$ be a coroot. Then $\p^{n_\a \a}\in H$.
\end{lemma}
\begin{proof}
Under these assumptions, $M$ is a maximal torus of $G$. We have to show that $\p^{n_\a \a}$ commutes with all of $\M$. To do this, it suffices to pass to an extension field where $G$ is split, then we can use the commutator formula (\ref{commutatorformula}) to deduce the result.
\end{proof}

Write $\La'$ for the lattice $\M/(\mu_n(M\cap K))$ and $\La$ for the sublattice $H/(\mu_n(M\cap {K}))$. Lemma \ref{inh} shows that $\La$ is finite index in $\La'$.
If $G$ is unramified, $\La'$ is canonically isomorphic to the cocharacter lattice $X_*(S)$.% and $\La$ is identified with the sublattice
%$$\La=\{\lambda\in X_*(S) \mid \sss(\p^\lambda)\in H\}=\{ x\in X_*(S) \mid B(x,y)\in n\Z\ \forall\ y\in X_*(S) \}.$$ The equivalence of the two given presentations is a consequence of the commutator formula \cite[3.1]{mcn}.

We will make one more assumption, which will be used in the construction of the principal series representations.
\begin{assumption}\label{abelian}
%There is a $W$-equivariant isomorphism $H/(M \cap K)\cong \mu_n\times \La$.
The quotient group $H/(M\cap K)$ is abelian.
\end{assumption}
%Such an isomorphism will not be canonical, but that does not concern us.
In the unramified case, the following lemma proves that this assumption is always satisfied.

\begin{lemma}\label{abelianlemma}
 If $G$ is unramified, then $H$ is abelian.
\end{lemma}
\begin{proof}
 Since $G$ is unramified, $M$ is abelian and $M=S(M\cap K)$. Now suppose that $h_1,h_2\in H$. Write $h_i=s_ik_i$ with $s_i\in \s$ and $k_i\in M\cap K$. The only nontrivial part is to show that $s_1$ and $s_2$ commute. But $s_i$ commutes with $\s\cap K$, so by \cite[Lemma 5.3]{mcn}, we're done.
\end{proof}

%One may wish to compare these conditions we have shown to hold in the unramified case with those appearing in \cite[D\'efinition 3.1.1]{li}.
% \begin{proposition}
%  If $G$ is unramified, then Assumption \ref{abelian} holds.
% \end{proposition}
% 
% 
% \begin{proof} There is a short exact sequence of abelian groups
% \[
%  0\to \mu_n \to H/(M\cap K) \to \La \to 0,
% \] where the middle term is abelian by Lemma \ref{abelianlemma}.
% Since $\La$ is a sublattice of $X_*(S)$, which we have already constructed a lift of to $\G$, we may define a section $\La\to H/(M\cap K)$ by $\la\to \p^\la$. This section induces the desired $W$-equivariant isomorphism. %Although not canonical, this isomorphism is $W$-equivariant, which is sufficient for our purposes.
% \end{proof}
The following elementary observation will often be used
\begin{lemma}\label{integratecharacter}
 Let $A$ be a compact topological group and $\chi\map{A}{\C^\times}$ a nontrivial homomorphism. Then $\int_A \chi(a)da=0.$
\end{lemma}

%%%%%%%%%%%%%%%%%%%%%%%%%%%%%%%%%%%%%%%%%%%%%%%%%
\section{Unramified representations}%%%%%%%%%%%%%r
%%%%%%%%%%%%%%%%%%%%%%%%%%%%%%%%%%%%%%%%%%%%%%%%%

Let $P$ be a minimal parabolic $F$-subgroup of $G$ containing $S$ and let $U$ be its unipotent radical. The quotient $P/U$ is canonically isomorphic to the Levi subgroup $M=Z_G(S)$.
Thinking of $M$ as a quotient of $P$ makes the constructions we make more canonical, for example $I(\chi)$ will not depend on the choice of a minimal parabolic subgroup up to canonical isomorphism.
% Let $N$ be the intersection of $\ker (\nu)$ as $\nu$ runs over all rational characters of $M$. Then we have $N\subset K$, $M=N.S$ and $N\cap S$ is a finite group.

Consider the complex algebraic variety
\[
 Y=\{\chi\in\Hom(H/(\M\cap K),\C^\times)\mid \chi(\zeta)=\epsilon(\zeta) \ \text{for all} \ \zeta\in\mu_n\}.
\]
There is a short exact sequence of groups\begin{equation}
\label{ses}
  0\to \mu_n \to H/(\M\cap K) \to \La \to 0,
 \end{equation} where the middle term is abelian by Assumption \ref{abelian}. As $\La$ is free, this sequence splits, so we obtain a noncanonical isomorphism $Y\cong \Spec(\C[\La])$. 

Consider a (complex) point $\chi\in Y$. We define the corresponding unramified representation $(\pi_\chi,i(\chi))$ of $\M$ as follows: Given $\chi$, we %turn it into a character of $\mu_n \times \La$ by letting $\mu_n$ act faithfully. In this way, $\chi$ defines 
inflate it to a one-dimensional representation of the subgroup $H$. We define $i(\chi)$ to be the induction of this representation from $H$ to $\M$. Note that $i(\chi)$ is finite dimensional.

We also construct an unramified principal series representation $I(\chi)$ of $\G$. First we use the canonical surjection $\P\rightarrow \M$ to consider $i(\chi)$ as a representation of $\P$. The representation $I(\chi)$ is defined to be the induction of $i(\chi)$ from $\P$ to $\G$.

Concretely, $I(\chi)$ is the space of all locally constant functions $f\map{\G}{i(\chi)}$ such that
\[
f(pg)=\delta^{1/2}(p)\pi_\chi(p)f(g)
\] for all $p\in\P$ and $g\in \G$ where $\delta$ is the modular quasicharacter of $\P$. The action of $\G$ on $I(\chi)$ is given by right translation.

The group $W_0$ acts on $\M$ by conjugation, and Lemma \ref{kerproj} shows that this descends to an action of $W$ on $\M$. As $W_0$ is a subgroup of $K$, this induces an action of $W$ on the subgroup $H$, under which the subgroup $M\cap K$ is stable. Hence we obtain an action of $W$ on $Y$. 
If $G$ is unramified, then $\La$ is a sublattice of $X_*(S)$ which allows us to choose the splitting $\la\mapsto \p^\la$ of (\ref{ses}). This choice makes the noncanonical isomorphism $Y\cong \Spec(\C[\La])$ Weyl group invariant.

We now define isomorphisms between the underlying vector spaces of $i(\chi)$ and $i(w\chi)$. As an induced representation, the underlying vector space of $i(\chi)$ is
\[
 \{ f\colon \M\to \C\mid f(mh)=\chi(h)f(m) \text{ for all } h\in H,m\in \M \}.
\]

Our choice of isomorphism $i(\chi)\xrightarrow{\sim} i(w\chi)$ is $f\mapsto {^wf}$, where
\[
 ^wf(m)=f(w^{-1}mw).
\]
These form a transitive system of isomorphisms.

% We may now define an explicit isomorphism between the underlying vector spaces of $i(\chi)$ and $i(w\chi)$, $f\mapsto{}^w\!f$ by $^w\!f(m)=f(wm w^{-1})$. Here, as $i(\chi)$ and $i(w\chi)$ are induced representations, we are realising them as spaces of functions on $\M$ which transform by $\chi$ and $w\chi$ respectively under left multiplication by $H$.

%This induces an action of $W$ on the category of representations of $\M$. In particular, we have constructed an explicit isomorphism between the underlying vector spaces of $i(\chi)$ and $i(w\chi)$ for any $\chi$ and any $w\in W$.

\begin{theorem}\label{kfixed}
The map $f\mapsto f(1)$ is an isomorphism between $I(\chi)^K$ and $i(\chi)^{\M\cap K}$. These are both
one-dimensional vector spaces.
\end{theorem}
\begin{proof}
The argument of \cite[Lemma 6.3]{mcn} applies in this case without change.
\end{proof}
The identification of the spaces $i(\chi)$ and $i(w\chi)$ constructed above can be construed as an action of $W$ on $i(\chi)$. Under this action, the subspace $i(\chi)^{\M\cap K}$ is invariant. Let us pick a non-zero vector $v_0$ in this subspace. By Theorem \ref{kfixed}, we choose a spherical vector $\phi_K^{(\chi)} \in I(\chi)^K$ for each $\chi$ in a $W$-orbit in a compatible manner such that $\phi_K^{(w\chi)}(1)=v_0$.

%%%%%%%%%%%%%%%%%%%%%%%%%%%%%%%%%%%%%%%%%%%%%
\section{Intertwining operators}%%%%%%%%%%%%%
%%%%%%%%%%%%%%%%%%%%%%%%%%%%%%%%%%%%%%%%%%%%%

For each positive root $\a$, there is an associated root subgroup $U_\a$ of $U$ with Lie algebra equal to $\operatorname{Lie}(G)_\a+\operatorname{Lie}(G)_{2\a}$.
For any $w\in W$, define $U_w$ to be the group equal to the product
\[
 \prod_{\a>0, w\a<0} U_\a.
\]
 This is also canonically isomorphic to the left quotient $(U\cap wUw^{-1})\lqt U$.

The (unnormalised) intertwining operators $T_w\map{I(\chi)}{I(w\chi)}$ are defined by
$$
(T_w f)(g)=%\prod_{\substack{\a>0 \\ w\a<0}}(1-x_a^{n_a})
\int_{U_w} f(w^{-1}ug)du.
$$ whenever this is absolutely convergent, and by a standard process of meromorphic continuation in general, for example following \cite[\S 7]{mcn}. It is a routine calculation to show that $T_w$ does indeed map $I(\chi)$ into $I(w\chi)$ as claimed. We denote by $X$ the open subset of $Y$ on which all the intertwining operators $T_w$ have no poles.
% RECALL WHAT Y IS!

When $w=s$ is a simple reflection, then we freely identify $U_s$ with the intersection of $U$ and the corresponding standard Levi subgroup $M_s$.

\begin{proposition}\label{fubini}
Suppose that $w_1$ and $w_2$ are two elements of $W$ such that $\ell(w_1w_2)=\ell(w_1)+\ell(w_2)$. Then the intertwining operators satisfy the identity $T_{w_1w_2}=T_{w_1}T_{w_2}$.
\end{proposition}

\begin{proof}
There is a measure preserving bijection from $U_{w_1}\times U_{w_2}$ to $U_{w_1w_2}$ given by $(u_1,u_2)\mapsto w_1u_2w_1^{-1}u_1$. The remainder of the proof is a standard manipulation involving Fubini's theorem.
\end{proof}

Let $B$ denote the fraction field of the coordinate ring $\mathcal{O}(Y)$.

\begin{theorem}\label{gk}
There exists a non-zero element $c_w(\chi)\in B$ such that
for all $\chi\in Y$,
$$
T_w \phi_K^{(\chi)} =  c_w(\chi) \phi_K^{(w\chi)}.
$$
\end{theorem}

\begin{remark}
In Theorem \ref{gkunramified} we will provide a more precise statement for unramified groups.
\end{remark}

\begin{proof}
By Proposition \ref{fubini}, we may assume without loss of generality that $w$ is a simple reflection $s$.
Also by meromorphic continuation, we may assume without loss of generality that the defining integral for $T_s$ converges.

The function $T_s\phi_K^{(\chi)}$ is guaranteed to be $K$-invariant, and hence by Theorem \ref{kfixed} is a scalar multiple of $\phi_K^{(s\chi)}$.
Thus $c_s(\chi)$ exists as a function on the open subset of $Y$ on which the intertwining operator $T_s$ has no poles. To compute $c_s(\chi)$, it suffices to evaluate $(T_s\phi_K^{(\chi)})(1)$.

There is a filtration on $U_s$ induced by a valuation of root datum. This consists of the data of a compact subgroup $U_{s,r}$ of $U_s$ for each $r\in \R$ with the property that $U_{s,r'}\subset U_{s,r}$ if $r'\leq r$. We let $C_r=U_{s,r}\setminus \cup_{r'<r}U_{s,r}$. First we collate some facts about these sets.

If $u\in C_r$ and $r>0$, then $s^{-1}u\in \mu_n \p^{r\a} U_s K$. Conjugation by $\p^{\a}$ sends $C_r$ to $C_{r+2}$, and there are only finitely many orbits of non-empty $C_r$ under the action of conjugation by $\p^\a$.

There is a decomposition of $U_s$ into the disjoint union
$$
U_s = (U_s\cap K) \bigcup \left( \bigcup_{r>0} C_r\right).
$$
We apply this decomposition to the domain of integration in the equation
$$
T_s\phi_K^{(\chi)}(1)=\int_{U_s} \phi_K^{(\chi)}( s^{-1} u ) du.
$$
The integral over $U_s\cap K$ is equal to $v_0$. For the integrals over the $C_r$, we use the fact that $T_s\phi_K^{(\chi)}(1)\in i(s\chi)^{\M\cap K}$ to see that only those $r$ for which $\p^{r\a}\in H$ can give a non-zero contribution. This provides an expression for $c_s(\chi)$ of the form
\[
c_s(\chi)=1+\sum_{r>0,\p^{r\a}\in H} d_r \chi(\p^{r\a})
\] for some constants $d_r$.

Note that by Lemma \ref{inh}, $\p^{n\a}\in H$.
In passing from the integral over $C_{r}$ to that over $C_{r+2n}$ via conjugation by $\p^{n\a}$, the integrand has been multiplied by $(\delta^{1/2}\chi)(\p^{2n \a})$ while the change of coordinates contributes a factor of $\delta(\p^{n \a})^{-1}$. Hence one integral is $\chi(\p^{2n\a})$ times the other. Thus our expression for $c_s(\chi)$ is actually a sum of geometric series, so is an element of $B$ as required.
%One simple way to see that $c_s(\chi)$ is non-zero is to take the limit as $x_\a$ tends to zero, when only the integral over $U_s\cap K$ survives.
%From the identity $su=bk$, we get $s(tut^{-1})=(sts)t^{-1} (tbt^{-1}) (tkt^{-1})$. So theintegral over $U_{m+2n_\a \a}$ is simply going to be $x_\a^{n_\a}$ times the integral over $U_m$. So we get a geometric series, completing the proof.
\end{proof}

In light of the above result, we now define a renormalised version of the intertwining operators. Let $\conj{T}_w=c_w(\chi)^{-1}T_w$. The upshot is that we have the equation
\begin{equation}\label{one}
 \conj T_w \phi_K^{(\chi)} = \phi_K^{(w\chi)}
\end{equation} as well as the following Proposition.

\begin{proposition}\label{whomomorphism}
For any $w_1,w_2\in W$, we have
\begin{equation}\label{wh}
\conj{T}_{w_1w_2}=\conj{T}_{w_1}\conj{T}_{w_2}. 
\end{equation}

\end{proposition}
\begin{proof}
The lattices $\La$ and $X_*(S)$ are commensurable, so
the set of characters $\chi$ with trivial stabiliser under the $W$-action is dense in $Y$. Hence it suffices
to prove the proposition for such $\chi$. As in \cite[Theorem 7.1]{mcn}, we use \cite[Theorem 5.2]{bz} to conclude that
$\dim \Hom_{\G} (I(\chi),I(w\chi))=1$. The proposition now follows from Theorem \ref{gk}.
\end{proof}

%%%%%%%%%%%%%%%%%%%%%%%%%%%%%%%%%%%%%%%%%%%%%%%
\section{Iwahori invariants}%%%%%%%%%%%%%%%%%%%
%%%%%%%%%%%%%%%%%%%%%%%%%%%%%%%%%%%%%%%%%%%%%%%%

The structure and results of this section closely follow \cite[\S 3.3]{chintaoffen}. Let $I$ be the Iwahori subgroup of $G$, maximal with respect to intersection with $P$ among all Iwahori subgroups contained in $K$.

\begin{proposition}
The dimension of the space of vectors in an unramified principal series representation $I(\chi)$ that are invariant under the Iwahori subgroup $I$, is equal to $|W|$.
\end{proposition}

\begin{proof}
One has $\P \lqt \G/I\cong W$. %[If $C$ is the alcove corresponding to $I$ and $A$ is the  apartment containing $C$, then we can use $P$ to map $gC$ into $A$, but only up to translation by an element of $X_*(S)$. The equivalence classes of alcoves under this translation action is canonically isomorphic to $W$.]
Thus $\dim(I(\chi)^I)\leq |W|$.

For any $w\in W$, we can define $\phi_w\in I(\chi)^I$ by $\phi_w(g)=\phi_K(g)$ if $g\in \P w I$ and $\phi_w(g)=0$ otherwise. These functions are obviously linearly independent so the proposition is proved.
\end{proof}

Let $w_0$ denote the longest element in $W$. Let $\phi_{w_0}^{(\chi)}$ be the function in $I(\chi)$ supported on $\P w_0 I$ and taking the value $v_0$ at $w_0$.

\begin{proposition}\cite[Lemma 5]{chintaoffen}\label{basis}
A basis for the space $I(\chi)^I$ can be given by the elements $T_w\phi_{w_0}^{(w^{-1}\chi)}$ as $w$ ranges over $W$.
\end{proposition}

\begin{proof}
Let us write $f_w$ for the function $T_w\phi_{w_0}^{(w^{-1}\chi)}$. Suppose that $f_w(v)\neq 0$ for some $v\in W$. By definition,
\[
 f_w(v)=\int_{U_w}\phi_{w_0}^{(w^{-1}\chi)} (w^{-1}uv) du.
\]
For this to be nonzero, it is necessary that $w^{-1} U v\cap \P w_0 I\neq\emptyset$. Since the group $I$ admits an Iwahori factorisation, we have the containment $\P w_0 I\subset \P w_0 \P$. Now $w^{-1} U v\subset Pw^{-1}PvP$ and since these are double cosets in a Tits system, this intersects $Pw_0P$ non-trivially only if $\ell(w^{-1})+\ell(v)\geq \ell(w_0)$. Furthermore, if there is equality $\ell(w^{-1})+\ell(v)= \ell(w_0)$, the intersection is non-empty if and only if $w^{-1}v=w_0$.

From the above considerations, it suffices to show that $f_w(ww_0)\neq 0$. 
Since the support of $\phi_{w_0}$ is $\P w_0I$, we need to understand when $w^{-1}uww_0\in Pw_0I$ with $u\in U_w$. 
We rewrite this as $w^{-1}uw\in P\cdot w_0 I w_0^{-1}$. As $w^{-1}uw\in U$ and the group $w_0Iw_0^{-1}$ admits and Iwahori factorisation, this only happens when $w^{-1}uw\in (w_0 I w_0^{-1})\cap U^-$. For such $u$, $w^{-1}uw w_0\in w_0 I$ and hence $\phi_{w_0}^{(w^{-1}\chi)}(w^{-1}uw w_0)=v_0$.

Therefore the above integral for $f_w(ww_0)$ reduces to an integral of $v_0$ over a compact subset of $U_w$ with positive measure, hence the integral is nonzero, as required.
\end{proof}

\begin{proposition}\label{sphericalexpansion}
The spherical function $\phi_K$ can be expanded as
\begin{equation*}
\phi_K=\sum_{w\in W} c_{w_0}( w^{-1} \chi) \conj T_w\phi_{w_0}^{(w^{-1}\chi)}.
\end{equation*}
\end{proposition}
\begin{proof}
By the previous proposition, there exist $d_w(\chi)$ such that \begin{equation}\label{dwchidefn}
\phi_K^{(\chi)}=\sum_{w\in W} d_w(\chi) \conj T_w\phi_{w_0}^{(w^{-1}\chi)}.
\end{equation} Let us apply $\conj T_u$ to this equation. Via (\ref{one}) and (\ref{wh}), we arrive at

$$ \sumw d_w(u\chi) \conj T_w \phi_{w_0}^{(w^{-1}u\chi)} = \sumw d_w(\chi) \conj T_{uw} \phi_{w_0}^{(w^{-1}u\chi)}.$$

Again using the fact we have a basis of $I(u\chi)^I$, we compare coefficients to obtain $d_w(\chi)=d_{uw}(u\chi)$. Thus it suffices to establish the value of $d_{w_0}(\chi)$. We now evaluate (\ref{dwchidefn}) at the identity.

$$\phi_K^{(\chi)}(1)=\sum_{w\in W} d_w(\chi)c_{w_0}(w\chi)^{-1} \int_{U_w} \phi_{w_0}^{(w^{-1}\chi)}(w^{-1}u) du. $$

As in the proof of Proposition \ref{basis}, $w^{-1}u\in Pw_0 I$ if and only if $w=w_0$ and $u\in U\cap I$. Thus only one term survives in this sum, and the surviving integral is the integral of $v_0$ over a set of measure one. We end up with $d_{w_0}(\chi)=c_{w_0}(w_0^{-1}\chi)$ which implies the result.
\end{proof}

%%%%%%%%%%%%%%%%%%%%%%%%%%%%%%%%%%%%%%%%%%%%%%%%%%%%%%%
\section{Whittaker functionals}%%%%%%%%%%%%%%%%%%%%%%%%
%%%%%%%%%%%%%%%%%%%%%%%%%%%%%%%%%%%%%%%%%%%%%%%%%%%%%%%

A group homomorphism $\psi:U\to \C^\times$ is said to be an unramified character if for each simple reflection $s\in W$, the restriction of $\psi$ to the inersection $U_s=U\cap M_s$ with the corresponding Levi subgroup $M_s$ is trivial on $U_s\cap K$ and non-trivial on all open subgroups of $U_s$ with a larger abelianisation than $U_s\cap K$. We fix one such unramified character $\psi$. By abuse of notation we also use $\psi$ to denote the corresponding group homomorphism from $U^-$ to $\C^\times$ obtained by precomposing $\psi$ with conjugation by $w_0$.

\begin{definition}
A Whittaker functional on a representation $(\pi,V)$ of $\G$ is a linear functional $W$ on $V$ such that $W(\pi(u)v)=\psi(u)v$ for all $u\in U$ and $v\in V$.
\end{definition}

\begin{theorem}\label{whittakerfunctionals}
There is an isomorphism between $i(\chi)^*$ and the space of Whittaker functionals on $I(\chi)$ given by $\la\mapsto W_\la$ with
$$
W_\la(\phi) = \la\left( \int_{U^-} \phi(uw_0) \psi(u) du \right).
$$
\end{theorem}
\begin{proof}
This follows from \cite[Theorem 5.2]{bz} and the Bruhat decomposition.
\end{proof}

Define $W^{(\chi)}:I(\chi)\to i(\chi)$ by
\[
 W^{(\chi)}(\phi) = \int_{U^-}\phi(uw_0) \psi(u) du.
\]
We use $R(g)$ to denote the action of $g\in\G$ on a function by right translation.
A group element $t\in \M$ is said to be dominant if $t(U\cap K)t^{-1}\subset K$.

% As $a$ runs through such a set of coset representatives, the vectors $\pi_\chi(a)v_0$ form a basis of $i(\chi)$. We will write $\la_a^{(\chi)}$ for the corresponding dual basis of $i(\chi)^*$ and let $W_a^{(\chi)}$ be the Whittaker functional corresponding to $\la_a^{(\chi)}$ under the bijection of Theorem \ref{whittakerfunctionals}. The functional $\la_a^{(\chi)}$ depends only on $a$ and not on the choice of a set of coset representatives including $a$.
% 
% 
% For any $a$ and any $w$, the composite $W_a^{(w\chi)}\circ \conj T_w$ is also a Whittaker functional on $I(\chi)$. Thus it can be expanded in any basis we have, so we define coefficients $\tau_{a,b}^{(w,\chi)}$ by
% $$
% W_a^{(w\chi)}\circ \conj T_w=\sum_b \tau_{a,b}^{(w,\chi)} W_b^{(\chi)}.
% $$
% 
% We need to know how these coefficients change under a change of coset representatives. For $h\in H$, we have
% $$
% \tau_{a,bh}^{(w,\chi)}=\chi(h) \tau_{a,b}^{(w,\chi)} \quad {\rm and} \quad \tau_{ah,b}^{(w,\chi)}=\frac{\tau_{a,b}^{(w,\chi)}}{(w\chi)(h)}.
% $$
% As a consequence, after choosing an extension of $\chi$ to a function on $\M$ satisfying $\chi(mh)=\chi(m)\chi(h)$ for $h\in H$, the quantity
% $$
% \widetilde\tau_{a,b}^{(w,\chi)}=\frac{(w\chi)(a)}{\chi(b)} \tau_{a,b}^{(w,\chi)}
% $$
% is independent of the choice of coset representatives, only depending on the cosets of $a$ and $b$ (and also depending on the choice of extension of $\chi$).

\begin{lemma}\cite[Lemma 7]{chintaoffen}\label{colemma7}
Let $t\in\M$. Then
\[
 W^{(\chi)}( R(t)\phi_{w_0}^{(\chi)}) = \begin{cases}
                                         \delta^{1/2}(t)\pi_\chi(w_0 t w_0^{-1}) v_0 & \text{if $t$ is dominant} \\
0 & \text{otherwise} \\
                                        \end{cases}
\]
\end{lemma}
\begin{proof}
Let $t^*=w_0 tw_0^{-1}$ and $u'=(t^*)^{-1}ut^*\in U^-$. Since the group $I$ admits an Iwahori factorisation, we have $u'w_0\in \P w_0 I$ if and only if $u'\in U^-\cap w_0 I w_0^{-1}$. Therefore
\[
\phi_{w_0}^{(\chi)} (uw_0 t)= \phi_{w_0}^{(\chi)} (t^* u' w_0) 
= \delta^{1/2}(t^*) \pi_\chi(t^*) \phi_{w_0}^{(\chi)} (u'w_0).
\]
Since the support of $\phi_{w_0}^{(\chi)}$ is $\P w_0 I$, our above remark shows that this expression vanishes unless $u'\in U^-\cap w_0 I w_0\inv$, which implies $u'w_0\in w_0 I$. Since $\phi_{w_0}^{(\chi)}$ is $I$-invariant on the right, $\phi_{w_0}^{(\chi)} (u'w_0)=\phi_{w_0}^{(\chi)} (w_0)$ under these circumstances. Therefore
\[
\phi_{w_0}^{(\chi)} (uw_0 t)= \begin{cases}
                                         \delta^{1/2}(t^*)\pi_\chi(t^*)v_0 & \text{if $u'\in U^-\cap w_0 I w_0^{-1}$} \\
0 & \text{otherwise}. \\
                                        \end{cases}
\]

Now we can compute
\begin{eqnarray*}
 W^{(\chi)}( R(t)\phi_{w_0}^{(\chi)}) &=& \int_{U^-} \phi_{w_0}^{(\chi)} (uw_0 t) \psi(u) du \\
&=& \delta^{1/2}(t^*) \pi_\chi(t^*)v_0 \int_{t^* (U^-\cap K)(t^*)^{-1}} \psi(u) du.
\end{eqnarray*}

This final integral vanishes unless $\psi$ is identically 1 on the subgroup being integrated over, which happens exactly when $t$ is dominant, by our assumption that $\psi$ is unramified. In this case, we get a factor of the volume of the group appearing, which is $\delta(t)=\delta^{-1}(t^*)$. This completes the proof.
% Since the group $I$ admits an Iwasawa decomposition, for any $u'\in U^-$, we have $u'w_0\in \P w _0I$ if and only if $u'\in U^-(O_F)$. We apply this to $u'=(t^*)^{-1}ut^*$, to obtain
% $$
% W_b^{(\chi)} ( R(t)\phi_{w_0}^{(\chi)} ) = \la_b^{(\chi)} \delta^{1/2}(t^*) \pi_\chi (t^*)v_0 \int_{t^* U^-(O_F)(t^*)^{-1}} \psi(u) du.
% $$
% The first factor vanishes unless $t^*\sim b$, when it is equal to $\delta^{1/2}(t^*)\chi(t^* b^{-1})$. The integral vanishes whenever $\psi$ is a non-trivial character on the group $t^* U^-(O_F)(t^*)^{-1}$. This occurs unless $t$ is dominant, in which case we get the relevant volume, namely $\delta(t)$, appearing as a factor.
\end{proof}

%%%%%%%%%%%%%%%%%%%%%%%%%%%%%%%%%%%%%%%%%%%%%%%%%%%%%%%%%%%%%%%
\section{Constructing the Chinta-Gunnells action}\label{saction}
%%%%%%%%%%%%%%%%%%%%%%%%%%%%%%%%%%%%%%%%%%%%%%%%%%%%%%%%%%%%%%%%

Let $\mathcal{E}$ be the holomorphic vector bundle on $Y$ whose fibre over a point $\chi\in Y$ is canonically isomorphic to $i(\chi)$. Its dual $\mathcal{E}^*$ is another holomorphic vector bundle on $Y$ whose fibre over a point $\chi$ is canonically isomorphic to $i(\chi)^*$. Let $\mathcal{F}$ be the holomorphic vector bundle on $Y$ whose fibre over a point $\chi$ is canonically isomorphic to the space of Whittaker functionals on $I(\chi)$.

Theorem \ref{whittakerfunctionals} provides us with an isomorphism between $\mathcal{E}^*$ and $\mathcal{F}$.

%On $Y$ we have holomorphic vector bundles $\underline{i(\chi)}$, $\underline{i(\chi)^*}$ and $\W$, whose fibres over a point $\chi\in Y$ are respectively $i(\chi)$, $i(\chi)^*$ and the space of Whittaker functionals on $I(\chi)$. Theorem \ref{whittakerfunctionals} furnishes us with an isomorphism between $\underline{i(\chi)^*}$ and $\W$.

Let $V$ be the open dense subset of $Y$ where the rational functions $c_w(\chi)$ have neither zeroes nor poles.% A character $\chi$ is said to be regular if its orbit under $W$ contains $|W|$ distinct characters.

For any $\chi\in V$, the space $\End(I(\chi))$ is one dimensional by \cite[Theorem 7.1(2)]{mcn}, which is valid in this more general context with the same proof. Hence this endomorphism space consists only of scalars. The map $\overline{T}_s^2$ is an endomorphism of $I(\chi)$ sending $\phi_K$ to itself, hence must be the identity. 

This fact allows us to define an action of the Weyl group $W$ on the total space of $\mathcal{F}|_V$.
For $w\in W$ and $\mathcal{W}$ a point of the total space of $\mathcal{F}|_V$, let
\[
 w\cdot \mathcal{W} = \mathcal{W} \circ \overline{T}_w.
\]
If $\mathcal{W}$ is a Whittaker functional on $I(\chi)$, then the composite $\mathcal{W}\circ \overline{T}_w$ is easily seen to be a Whittaker functional on $I(w^{-1}\chi)$.

Via the isomorphism from $\mathcal{F}$ to $\mathcal{E}^*$, we transport the $W$ action on the total space of $\mathcal{F}|_V$ to a $W$ action on the total space of 
$\mathcal{E}^*|_V$. This is not related to any more naive action on $\mathcal{E^*}$. From our point of view, the Chinta-Gunnells action will be the induced $W$-action on the space of algebraic global sections $\Gamma(V,\mathcal{E}^*)$. We will explain in the final section of this paper why this agrees with the action constructed in \cite{cg}.

Note that the fact that the $W$-action is algebraic has not yet been justified. We postpone the verification of this fact until the end of section \ref{parttwo}.

\section{Formal computation of the Whittaker function}
%%%%%%%%%%%%%%%%%%%%%%%%%%%%%%%%%%%%%%%%%%%%%%%%%%%%%%%%%%%%%%%

Our aim is to compute the general spherical Whittaker function
\[
 \W(g)=W^{(\chi)}(R(g)\phi_K).
\]
This is
% There is another basis of $i(\chi)^*$ in the unramified case that is more amenable to calculation than the one we have so far considered. 
% It is parametrised by extensions of $\chi$ to a character $\tilde\chi$ of $X_*(S)$, and only depends on the choice of a lift of the lattice $X_*(S)$ to a subgroup of $\G$. In this way, given a coset representative as in the previous section of the form $\p^\la$, we can meaningfully talk about $\tilde\chi(a)$.
% 
% In general, we have to be more circumspect, and do not get anything approaching a distinguished choice of basis. As in the previous section, we consider an extension $\tilde\chi$ of $\chi$ to a function on $\M$ satisfying $\tilde\chi(mh)=\tilde\chi(m)\chi(h)$ for $m\in\M$ and $h\in H$. We consider the functional $\la_{\tilde\chi}$ on $i(\chi)$ defined by $\la_{\tilde\chi}(\pi_\chi(a)v_0)=\tilde\chi(a))$, and will compute the corresponding Whittaker function.
% 
% In the unramified case, choosing $\tilde\chi(\p^\la)=\tilde\chi(\la)$ yields a basis of $i(\chi)^*$ as $\tilde\chi$ runs over the set of all extensions of $\chi$ to $X_*(S)$. Let us now fix for once and all a particular extension $\tilde\chi$ of $\chi$. By abuse of notation, we will simply write $\chi$ for this extension throughout.
% 
% The Whittaker function which we we are aiming to compute is the function
% $$
% \W(g)=W_{\la_\chi}^{(\chi)}(R(g)\phi_K).
% $$
an $i(\chi)$-valued function on $\G$ satisfying
$$
\W(\zeta u g k)=\zeta \psi(u) \W(g)
$$
for all $\zeta\in \mu_n$, $u\in U$, $g\in \G$ and $k\in K$. The Iwasawa decomposition takes the form $G=UMK$, so in order to compute $\W$, it suffices to know the
values taken by $\W$ on $\M$. This is what we shall concentrate our efforts on.

\begin{theorem}\label{formal}
 Suppose $\la\in i(\chi)^*$ and $t\in \M$. The Whittaker function $\W(t)$ vanishes unless $t$ is dominant. If $t$ is dominant, then we have
$$
\la(\W(t))=\delta^{1/2}(t) \sum_{w\in W} c_{w_0} (w^{-1}\chi ) (w\cdot \la)( \pi_{w^{-1}\chi}(w_0 t w_0^{-1})v_0 ).
$$
\end{theorem}
\begin{proof}
Using Lemma \ref{sphericalexpansion}, the fact that $T_w$ is an intertwining operator and the definition of the $W$-action, we compute
\begin{eqnarray*}
 \la(\W(t))&=& \la(W^{(\chi)} R(t) \phi_K ) \\
&=&
\sum_{w\in W} c_{w_0}(w^{-1}\chi)\la( W^{(\chi)}(R(t)\overline{T}_w (\phi_{w_0}^{(w^{-1}\chi)} ))) \\
&=&\sum_{w\in W} c_{w_0}(w^{-1}\chi)\la( W^{(\chi)}(\overline{T}_w (R(t)\phi_{w_0}^{(w^{-1}\chi)} ))) \\
&=&\sum_{w\in W} c_{w_0}(w^{-1}\chi)(w\cdot \la)( W^{(w^{-1}\chi)} ( R(t) \phi_{w_0}^{(w^{-1}\chi)}) )
\end{eqnarray*}
and the result now follows from Lemma \ref{colemma7}.
\end{proof}
\section{Setting up the explicit computation}\label{parttwo}
%%%%%%%%%%%%%%%%%%%%%%%%%%%%%%%%%%%%%%%%%%%%%%%%%%%%%%%%%%%%%%%%

Let $\Ga$ be a set of left coset representatives for $H$ in $\M$. 
A particularly nice, but still noncanonical, choice can be made whenever $G$ is unramified and this is carried out when $G$ is split in Section \ref{comparison}.
%When $G$ is unramified, it is possible to choose these coset representatives all of the form $\varpi^\la$ for some $\la\in X_*(S)$. Such a choice would be made for convenience, as opposed to necessity.
The vectors $\{\pi_\chi(a)v_0\}_{a\in\Ga}$ form a basis of $i(\chi)$. Incidentally, this yields a trivialisation of the vector bundle $\mathcal{E}$.

Let $\la_a^{(\chi)}$ denote the corresponding dual basis of $i(\chi)^*$, and let $W_a^{(\chi)}$ be the Whittaker functional corresponding to $\la_a^{(\chi)}$ under the bijection of Theorem \ref{whittakerfunctionals}.

We now introduce the change of basis coefficients as in \cite[\S I.3]{kp}.

As has already been mentioned, for any $a\in\Ga$ and any $w\in W$, the composition $W_a^{(w\chi)}\circ \conj T_w$ is a Whittaker functional on $I(\chi)$. We expand it in the basis $\{W_b^{(\chi)}\}_{b\in \Ga}$, defining coefficients $\tau_{a.b}^{(w,\chi)}$ by
\begin{equation}\label{taudefn}
 W_a^{(w\chi)}\circ \conj T_w=\sum_{b\in \Ga} \tau_{a,b}^{(w,\chi)} W_b^{(\chi)}.
\end{equation}

Any explicit computation of metaplectic Whittaker functions reduces via Theorem \ref{formal} to computing the Weyl group action, and hence these coefficients.

% In this section we lay the groundwork for the explicit calculation for unramified $G$ that will follow. As a side effect, we will also be able to provide a proof of the missing claim that $\tau_{a,b}^{(w,\chi)}\in B$.

Let $K_1$ be an open compact subgroup of $G$ normalised by $W$ and admitting an Iwahori factorisation with respect to $P$. 
For each $b\in \M$, define $f_b\in I(\chi)^{K_1}$ to be a function in this space supported on $\P w_0 K_1$ and taking the value $\pi_\chi(b)v_0$ at $w_0$.

\begin{lemma}
We have $W_a^{(\chi)}(f_b)=0$ unless $aH=bH$, in which case $W_a^{(\chi)}(f_a)=|U^-\cap K_1|$.
\end{lemma}
\begin{proof}
Suppose that $u\in U^-$ and $f_b(uw_0)\neq 0$. The support of $f_b$ is $\P w_0 K_1$ and since $K_1$ is invariant under conjugation by $w_0$ this implies that $u\in \P K_1$. As $K_1$ has the Iwahori factorisation $K_1=(\P\cap K_1)(U^-\cap K_1)$, we have $u\in \P(U^-\cap K_1)$. As $U^-\cap \P=\{1\}$, this forces $u\in U^-\cap K_1$.

Therefore
\[
 W_a^{(\chi)}(f_b)=\la_a^{(\chi)}\left(\int_{U^-} f_b(uw_0)\psi(u) du\right) = \la_a^{(\chi)}\left(\int_{U^-\cap K_1} \pi_\chi(b)v_0\psi(u) du\right).
\]
By the definition of $\la_a^{(\chi)}$, this is zero unless $aH=bH$ and when $a=b$, we are integrating the constant 1 over the group $U\cap K_1$ so we get the measure $|U\cap K_1|$ as the answer.
\end{proof}

The following corollary is immediate.

\begin{corollary}
$$
\tau_{a,b}^{(w,\chi)}=\frac{ (W_a^{(w\chi)}\circ \conj T_w)(f_b) }{ |U^-\cap K_1| }.
$$
\end{corollary}

Now let us restrict our attention to the case where 
$w=s=s_\a$, a simple reflection corresponding to the simple coroot $\a$ of $G'$. From our definitios, the above corollary is equivalent to
\begin{equation}\label{cor9.2}
\tau_{a,b}^{(s,\chi)}=\frac{1}{ c_s(\chi) |U^-\cap K_1| }\la_a^{(s\chi)}\left (\int_{U^-}\int_{U_s} f_b(s^{-1}nuw_0)dn\psi(u)du \right).
\end{equation}

\begin{lemma}\label{fundamental}
Suppose that $n\in U_s$ is not equal to the identity, and $u\in U^-$. Then there is a unique $n'\in U_s^{-}$ such that $p(n):=s^{-1}nn'\in \P$. Furthermore, $s^{-1}n u w_0\in \P w_0 K_1$ if and only if $u=n'u'$ with $u'\in U^-\cap K_1$.
\end{lemma}
\begin{proof}
 The first claim is an immediate consequence of the Bruhat decomposition in the rank one group $M_s$. For the latter claim, write $u=n'u'$. Then $s^{-1}n u w_0\in \P w_0 K_1$ if and only if
$p(n)u' \in \P K_1$ and this set is equal to $\P (U^-\cap K_1)$ since $K_1$ admits an Iwasawa decomposition. This finishes the proof.
\end{proof}

%$$
%\begin{pmatrix}
%0 & -1 \\
%1 & 0
%\end{pmatrix}
%\begin{pmatrix}
%1 & x \\
%0 & 1
%\end{pmatrix}=
%\begin{pmatrix}
%1/x & -1 \\
%0 & x
%\end{pmatrix}
%\begin{pmatrix}
%1 & 0 \\
%1/x & 1
%\end{pmatrix}
%$$
%So we see that for any $n\in U_s$, there exists $n'\in U_s^-$ and $p(n)\in \B_s$ with $s^{-1}n=p(n)n'$. In the $SL_2$ case, this same identity holds with the section ${\bf s}$ applied to all matrices, as seen with the Kubota cocycle. For $SU_3$ case, the analagous statement appears false, and more work is needed.

As an immediate consequence of this lemma, we can simplify (\ref{cor9.2}) to leave ourselves with an expression for the coefficient $\tau_{a,b}^{(s,\chi)}$ with only one integral, namely
\begin{equation}\label{taurankone}
\tau_{a,b}^{(s,\chi)}= c_s(\chi)^{-1} \int_{U_s}\la_a^{(s\chi)} f_b( p(n) w_0 )  \psi(n') dn
\end{equation}
Since the subset $\{1\}\subset U_s$ where the integrand is not defined is of measure zero, it may safely be ignored.

\begin{proposition}
 The coefficient $\tau_{a,b}^{(w,\chi)}$ is a rational function on $Y$.
\end{proposition}
\begin{proof}
Since the simple reflections generate $W$, it suffices to consider the case where $w=s$ is a simple reflection. Now we may use (\ref{taurankone}) and our strategy is to argue in an analogous manner to the proof of Theorem \ref{gk}.

We use the same decomposition
$$
U_s = (U_s\cap K) \bigcup \left( \bigcup_{r>0} C_r\right).
$$
of the domain of integration in (\ref{taurankone}) induced by the valuation of root datum.

If $n\in C_r$ with $r>0$, then $n'\in K$ and hence $\psi^{-1}(n')=1$. Furthermore $p(\p^{-n\a}n\p^{n\a})=\p^{2n\a}p(n)\pmod{U}$. Thus in passing from $C_r$ to $C_{r+2n}$, the integrand in (\ref{taurankone}) %$\la_a^{(s\chi)} f_b( p(n) w_0 )  \psi^{-1}(n')$ 
is multiplied by $(\delta^{1/2}\chi)(\p^{2n\a})$ while the volume of $C_{r+2n}$ is $\delta^{-1}(\p^{n\a})$ times that of $C_r$.

Therefore the integral over $C_{r+2n}$ is equal to $\chi(\p^{2n\a})$ times the integral over $C_r$.
As in the proof of Theorem \ref{gk}, computing the integral (\ref{taurankone}) reduces to adding an integral over a compact piece times a geometric series. The compact piece is dealt with in exactly the same manner as before, completing the proof.
\end{proof}

%%%%%%%%%%%%%%%%%%%%%%%%%%%%%%%%%%%%%%%%%%%%%%%%%%%%%%%%%%
\section{Digression on $SU_3$}\label{digression}%%%%%%%%%%
%%%%%%%%%%%%%%%%%%%%%%%%%%%%%%%%%%%%%%%%%%%%%%%%%%%%%%%%%%%
From now until the end of the paper we will assume that $G$ is an unramified group. 
For any simple reflection $s$, let $G_s$ be the simply connected cover of the derived group of $M_s$. The group $G_s$ is a simply-connected, semisimple unramified group of rank one, and such groups are completely classified. There are two possibilities, either $G_s$ is isomorphic to $\Res_{E/F}SL_2$ for an unramified extension $E$ of $F$, or is isomorphic to $\Res_{E/F}SU_3$, where the special unitary group $SU_3$ over $E$ is defined in terms of an unramified quadratic extension $L$ of $E$, which again is unramified over $F$.

Of these two possibilities, the group $SL_2$ will be familiar to most readers. We pause to collate some facts about the less well-known $SU_3$ that will prove to be of use later on.

We use $\conj z$ to denote the image of $z$ under the non-trivial element of $\Gal(L/E)$. 
The special unitary group $SU_3(E)$ is defined to be the subgroup of $SL_3(L)$ preserving the Hermitian form $x_1\conj{y_3}+x_2\conj{y_2}+x_3\conj{y_1}$.
Explicitly, if $J$ is the matrix with ones on the off-diagonal and zeroes elsewhere, then $X\in SL_3(L)$ is in $SU_3(E)$ if and only if $^t\conj{X}JX=J$.

These coordinates are chosen such that the intersection of $SU_3(E)$ with the set of upper-triangular matrices constitutes a Borel subgroup. Its unipotent radical consists of all matrices of the form
\begin{equation}\label{su3udefn} u=\begin{pmatrix}
1  & x & y  \\
0 & 1 & -\conj{x}\\
0 & 0 & 1
\end{pmatrix}
\end{equation} where $x$ and $y$ are elements of $L$ with $x\conj x +y+\conj y=0$.

We take our maximal compact subgroup $K$ to be the subgroup consisting of all matrices in $SU_3(E)$ with entries in $O_L$.

Let $v$ denote the valuation on $E$. For any $r\in \R$, the set of $u\in U$ with $v(y)\geq r$ forms a subgroup of $U$. (This is the filtration induced by a valuation of root datum in Bruhat-Tits theory introduced in the proof of Theorem \ref{gk}). Let us denote this subgroup by $U_r$.

%A particular aspect that will require some care, is that the fibres of the map $u\mapsto \p^{-2m} y$ from $U_{2m}\setminus U_{2m-1}$ to $O_L^\times$ do not all have the same volume. Namely the volume of a fibre over a point $z$ with $z+\conj z \in O_L^\times$ is $q+1$ times the volume of the fibre over a point $z$ with $z+\conj z\in \p O_L$.

The following equation is fundamental, and explicitly realises the first part of Lemma \ref{fundamental} in $SU_3(E)$.
\begin{equation}\label{su3fundamental}
\begin{pmatrix}
0 & 0 & 1 \\
0 & -1 & 0\\
1 & 0 & 0
\end{pmatrix}\begin{pmatrix}
1  & x & y  \\
0 & 1 & -\conj{x}\\
0 & 0 & 1
\end{pmatrix}=\begin{pmatrix}
1/\conj{y} & x/y & 1 \\
0 & \conj{y}/y & \conj{x}\\
 0 & 0 & y
\end{pmatrix}\begin{pmatrix}
1 &0 & 0 \\
\conj{x}/\conj{y} & 1 & 0\\
\conj{y}^{-1} & -x/y &  1
\end{pmatrix}^{-1}
\end{equation}
In the next section we will see how to lift this to an equation in the metaplectic cover.

%Let us say that an element $u\in U$ is of type I if $v(y)=2v(x)$ and is of type II otherwise.

Let $\a$ be the positive generator of $X_*(S)$. We can, and do, write $\a=\a_1+\a_2$ where $\a_1$ and $\a_2$ are simple coroots of $SL_3$.

%%%%%%%%%%%%%%%%%%%%%%%%%%%%%%%%%%%%%%%%%%%%%%%
\section{The descent process}%%%%%%%%%%%%%%%%%%
%%%%%%%%%%%%%%%%%%%%%%%%%%%%%%%%%%%%%%%%%%%%%%%

Since we are only working with unramified groups, Galois descent for an unramified extension of local fields shall be the only descent process that shall concern us. %We will now give a discription of the construction of the metaplectic group $\G$ in a manner that clarifies the relationship with such Galois descent. This overview is described more precisely in \cite[\S 12.11]{bd}.

%First one constructs the central extension of $G(F)$ by $K_2(F)$. Then one pushes forward by the tame symbol $K_2(F)\rightarrow k^\times$ to arrive an an extension of $G$ by $k^\times$. One finally pushes forward this extension via the operation of raising to the $(q-1)/n$-th power to obtain the metaplectic group $\G$.

%Now suppose that $E$ is a degree $d$ unramified extension of $F$. This description of $\G$ is particularly amenable to descent and shows
%that $\G$ can be realised via descent as a subgroup of a group $\widetilde{G(E)}$ which is a central extension of $G(E)$ by the group of $n\frac{q^d-1}{q-1}$-th roots of unity. In fact this description has already been implicitly used in the discussion following Assumption \ref{assumption1} to justify the use of this assumption in the unramified case.

We will first consider how descent behaves under restriction of scalars for semisimple groups.
Let $E$ be an unramified field extension of $F$, ${\bf G'}$ be a reductive group over $F$ and ${\bf G}=\Res_{E/F}{\bf G'}$.
 %The answer we will get is not surprising. Namely, if $G=\Res_{E/F}G'$, then the metaplectic group $\G$ is isomorphic to the central extension of $G(F)=G'(E)$ by $\mu_n$ obtained by considering $G'$ as an algebraic group over $E$.

 Let $\Gamma$ be the Galois group $\Gal(E/F)$. Let $T'$ be a maximal torus of ${\bf G}'$ and $T=\Res_{E/F}(T')$, which is a maximal torus of ${\bf G}$. As Galois modules we have $X_*(T)=X_*(T')\otimes \Z[\Gamma]$.

Consider $T'(E)=T(F)\hookrightarrow T(E)$. At the level of cocharacter lattices this corresponds to the diagonal embedding $X_*(T')\hookrightarrow X_*(T')\otimes\Z [\Gamma]=X_*(T)$, namely $y\mapsto \sum_{\gamma\in \Gamma} y\otimes \ga$. Let us write $Q'$ for the restriction of $Q$ to the image of $X_*(T')$ in $X_*(T)$.

The precise claim about the behaviour of the central extension under descent in this situation is the following:

\begin{proposition}\cite[Proposition 12.9]{bd}
The central extension of ${\bf G}(F)$ by $\mu_n$ associated with the quadratic form $Q$ and the central extension of ${\bf G'}(E)$ by $\mu_n$ associated with the quadratic form $Q'$ are isomorphic.
\end{proposition}

The other descent calculation we need to study in detail is descent from $SL_3$ to $SU_3$, which will enable us to perform computations in the metaplectic cover of $SU_3$. Our strategy is to realise  $\widetilde{SU_3(E)}$ as a subgroup of the $n(q+1)$-fold cover of $SL_3(L)$, with the same quadratic form characterising the extension in each case. There is an explicit cocycle for the cover of $SL_3(L)$ given to us by Banks, Levi and Sepanski. Their result \cite[Theorem 7]{bls}, together with the equations appearing in its proof provide an algorithmic method to multiply in $\widetilde{SL_3(L)}$.
It provides us with a section $\sss$ and a 2-cocycle $\sigma$ for which multiplication is given by $\sss(g_1g_2)=s(g_1)s(g_2)\sigma(g_1,g_2)$. We caution the reader than upon restriction of $\sss$ to $SU_3(E)$, the image does not lie in the $n$-fold cover $\widetilde{SU_3(E)}$.

Our strategy for circumventing this problem to find explicit elements of $\widetilde{SU_3(E)}$ is to use Theorem \ref{unipotentsplitting} which states that all unipotent subgroups are uniquely split in central extensions. We will use this fact to lift the identity (\ref{su3fundamental}) into an identity in the metaplectic cover.

By construction the section $\sss$ canonically splits the group $U$ of upper-triangular unipotent matrices in $SL_3$. Hence, the splitting of the lower-triangular unipotent subgroup $U^-$ must be given by $u\mapsto \sss(w_0)\sss(w_0uw_0)\sss(w_0)$, where $w_0=\left(\begin{smallmatrix}
0 & 0 & 1 \\ 
0 & -1 & 0 \\
1&  0 & 0
\end{smallmatrix}\right)$.

Let $$n_1=\begin{pmatrix}
1  & x & y  \\
0 & 1 & -\conj{x}\\
0 & 0 & 1
\end{pmatrix}\quad\text{and}\quad n_2=\begin{pmatrix}
1 & -x/y & \conj{y}^{-1} \\
0 & 1 & \conj{x}/\conj{y} \\
0 & 0 &  1
\end{pmatrix}.$$
Using the results of Banks Levi and Sepanski referenced above, it may be computed that $\sigma(n_1w_0n_2,w_0)=(x,\conj{y}/y)^{Q(\a)}$ and $\sigma(w_0,p(n))=(y,\conj y)^{Q(\a)}$. Thus we have $$\sss(w_0)\sss(n_1)\sss(w_0)\sss(n_2)\sss(w_0)=(x,y/\conj{y})^{Q(\a)}(y,\conj{y})^{Q\a)}\sss(p(n)).$$

There is a subtlety that needs to be taken care of. According to our construction in Section \ref{metaplectic}, the choice of representative for the simple reflection $s$ is not the same as the element $w_0$ we have been using so far in this computation.
Instead, it differs by a factor of $\theta^\a$ where $\theta\in L$ is such that $|\theta|=1$ and $\theta+\conj{\theta}=0$.

 Now let us write $y=a\theta\p^m$. Then after one more (simpler) cocycle computation, we find that
\begin{equation}\label{su3fundm}
 s^{-1}n n' =  (x,y/\conj y)^{Q(\a)}(y,\conj y)^{Q(\a)}(a,\p)^{-mQ(\a)}a^\a \p^{m\a} \pmod U.
\end{equation}

%Let $b=\p^\mu$. The commutator formula gives $[v,\p^\mu]=(\p,v)^{B(\a_1,\mu)}(\p,\conj{v})^{B(\a_2,\mu)}$. Since $B$ is Galois-invariant and $\mu$ is Galois-fixed, this commutator becomes $(\p,v\conj{v})^{B(\a,\mu)/2}$.

%Now there are two .

%Case 1 is where $v(y)=2v(x)$. Then we get $(\p,v\conj{v})^{kQ(\a)+B(\mu,\a)/2}$ where $k=m/2$. Case two is where $v(y)\neq 2v(x)$. Then we get $(\p,a)^{mQ(\a)+B(\mu,\a)}$ where $a\in O_F^\times$. This gives the desired roots of unity to allow us to complete the $SU_3$ calculation.

%Note that $Q(\a)$ divides$B(\a,\mu)/2$.

%%%%%%%%%%%%%%%%%%%%%%%%%%%%%%%%%%%%%%%%%%%%%%%%%%%%%%%%%%%%%%%%%%%%%%%%%%
\section{Gindikin-Karpelevic formula}\label{gkformula}%%%%%%%%%%%%%%%%%%%%
%%%%%%%%%%%%%%%%%%%%%%%%%%%%%%%%%%%%%%%%%%%%%%%%%%%%%%%%%%%%%%%%%%%%%%%%%%

Recall that for a simple reflection $s$, $G_s$ is defined to be the simply connected cover of the derived group of the corresponding Levi subgroup $M_s$. Let $q$ denote the cardinality of the residue field of $E$, where $E$ is as in the classification of $G_s$. For a simple coroot $\a$, let $x_\a^{n_\a}=\chi(\p^{n_\a\a})$, $n_\a=n/\gcd(n,Q(\a))$ and $\epsilon=(-1)^{n_\a}$. We have the following refinement of Theorem \ref{gk}.
Since $c_{w_1w_2}(\chi)=c_{w_1}(w_2\chi)c_{w_2}(\chi)$ whenever $\ell(w_1w_2)=\ell(w_1)+\ell(w_2)$, this is enough to evaluate $c_w(\chi)$ for all Weyl group elements $w$.

\begin{theorem}\label{gkunramified}
Suppose that $G$ is unramified, and the simple reflection $s$ corresponds to the simple coroot $\a$. Then
$$
c_s(\chi)=\begin{cases}
 \displaystyle\frac{1-q^{-1}x_\a^{n_\a}}{1-x_{\a}^{n_\a}} & \text{if\ }\  G_s\cong SL_2(E) \\
\displaystyle\frac{( 1+\epsilon q^{-1}x_\a^{n_\a} )( 1-\epsilon q^{-2}x_\a^{n_\a} )}{1-x_\a^{2n_\a}} & \text{if\ } \ G_s\cong SU_3(E) \\
% \frac{( 1+q^{-1}x_\a^{n_\a} )( 1-q^{-2}x_\a^{n_\a} )}{1-x_\a^{2n_\a}} & \text{if\ } \ G_s\cong SU_3(E) \text{  and  } n_\a \text{ is odd} \\
% \frac{( 1-q^{-1}x_\a^{n_\a} )( 1+q^{-2}x_\a^{n_\a} )}{1-x_\a^{2n_\a}} & \text{if\ } \ G_s\cong SU_3(E) \text{  and  } n_\a \text{ is even}
 \end{cases}
$$
\end{theorem}

\begin{proof}
The author has already written a proof in the case where $G_s\cong SL_2(E)$ \cite[Theorem 6.4]{decom}, so we will not repeat the argument here. When $G_s\cong SU_3(E)$, we present the
proof as a warm up for the more challenging computation of $\tau_{a,b}^{(s,\chi)}$ that will appear in Section \ref{su3rankone}. We follow the strategy from the proof of Theorem \ref{gk}. Hence we have to evaluate the integral in the formula

\begin{equation}\label{cschi}
c_s(\chi)v_0=\int_{U_s} \phi_K(s^{-1}u) du.
\end{equation}

The integral over $U_s\cap K$ is trivially equal to $v_0$. We focus our attention on the remainder of the integral.

As in (\ref{su3udefn}), we will write \begin{equation} u=\begin{pmatrix}
1  & x & y  \\
0 & 1 & -\conj{x}\\
0 & 0 & 1
\end{pmatrix}
\end{equation} where $x$ and $y$ are elements of $L$ with $x\conj x +y+\conj y=0$.
Note that this last equation implies that $y\in O_L$ if and only if $u\in K$.

Let $-v(y)=m$. First let us suppose that $2v(x)=v(y)$. Then $m$ is even, let $k=m/2$. Recall $y=a\theta\p^{-m}$. Then using the formula
(\ref{hilbert}) for an unramified Hilbert symbol, the product $(x,y/\conj{y})(y,\conj{y})(a,\p)^{-m}$ is equal to
\[
 \left( \frac{\conj{a}}{a} \right)^k a^m \conj{a}^{-m} a^{-m} = (a\conj{a})^{-k}
\] where by abuse of notation take the reductions modulo $\p$ of all $a$'s and $\conj{a}$'s in the above equation, and include $\F_{q^2}^\times$ into $\C^\times$.

If instead $2v(x)\neq v(y)$, then $y/\conj{y}\equiv -1 \pmod{\p}$, so under these circumstances, $(x,y/\conj{y})=1$, using $q\equiv 1\pmod{2n}$. Thus the root of unity $(x,y/\conj{y})(y,\conj{y})(a,\p)^{-m}$ is equal to $(\p,a)^m$.

Therefore, for all $u\in C_m$, we have $\phi_K(s^{-1}u)=(a\conj a,\p)^{mQ(\a)/2} \phi_K(\p^{m\a})$. So (\ref{cschi}) becomes
\[
c_s(\chi)=v_0 +\sum_{m=1}^\infty \int_{C_m} (a\conj a,\p)^{mQ(\a)/2} \phi_K (\p^{m\a}). 
\]

From this equation it is evident that the expression does not depend on the ambient group $G$, only on $G_s$. If we choose to take $G=SU_5$, then there is a coroot $\beta$ such that $B(\a,\beta)=Q(\a)$. By (\ref{commutatorformula}), the only powers of $\p^\a$ which lie in $H$ are powers of $\p^{n_\a \a}$. So all terms in the above sum where $m$ is not divisible by $n_\a$ must be zero.

%oh, crap if there are spare roots of unity floating around here they will start to surface. or is it that they're all hilbert symbols so will vanish. but still, need more care, even though they all vanish.

First consider the case when $n_\a$ is odd. Then $n_\a$ divides $2k$ if and only if it divides $k$. For such $m=2k$, we have $(a\conj a,\p)^{mQ(\a)/2}=(a\conj a,\p)^{kQ(\a)}$ and this exponent is divisible by $n$, hence the root of unity is 1. Therefore the contribution to (\ref{cschi}) from the terms where $m$ is even is
\begin{align*}
\sum_{l=1}^\infty \operatorname{vol}(C_{2ln_\a}) (q^{-2} x_\a)^{2ln_\a} v_0&=
\sum_{l=1}^\infty q^{4ln_\a}(1-q^{-3}) (q^{-2} x_\a)^{2ln_\a}v_0 \\ &= (1-q^{-3}) \frac{ x_\a^{2n_\a} }{ 1-x_\a^{2n_\a} }v_0.\end{align*}
When $m$ is odd and divisible by $n_\a$, write $m=(2l+1)n_\a$. Necessarily $v(y)\neq 2v(x)$ so the root of unity is simply $(\p,a)^{mQ(\a)}$, again the exponent is divisible by $n$ and so the root of unity is 1. So the contribution to (\ref{cschi}) from the terms where $m$ is odd is
\begin{align*}
\sum_{l=0}^\infty \operatorname{vol}(C_{(2l+1)n_\a}(q^{-2}x_\a)^{(2l+1)n_\a}v_0&=
\sum_{l=0}^\infty (q-1) q^{(2l+1)n_\a-2} (q^{-2}x_\a)^{(2l+1)n_\a}v_0 \\
&=q^{-2}(q-1) \frac{x_\a^{n_\a}}{1-x_\a^{2n_\a}}v_0.
\end{align*}

This completes the proof when $n_\a$ is odd.
Now let us assume that $n_\a$ is even. 

As before, we only need to consider the integral over $C_m$ where $m$ is divisible by $n_\a$. This implies that $m$ is even. Let $m=2k$.

If $n_\a$ divides $k$ then $(a\conj a,\p)^{kQ(\a)}=1$. The contribution to (\ref{cschi}) from such terms is again
\[
 (1-q^{-3}) \frac{ x_\a^{2n_\a} }{ 1-x_\a^{2n_\a} }v_0.\]

Now suppose that $m=2k$ is such that $n_\a$ divides $m$ but does not divide $2k$.

The fibres of the map from $C_m$ to $\mathbb{F}_{q^2}^\times$ given by $u\mapsto a\pmod{\p}$ do not all have the same size. In particular the volume of the fibre over a point not in $\F_q$ is $q+1$ times the volume of the fibre over a point in $\F_q$. 

Note that for $a\in\F_q^\times$, we have $(a\conj a,\p)^{kQ(\a)}=1$ while on the larger group $\F_{q^2}^\times$, the homomorphism $a\mapsto (a\conj a,\p)^{kQ(\a)}$ is nontrivial.

We compute
\begin{align*}
\int_{C_m} (a\conj a,\p)^{kQ(\a)} &= q^{2m-3}\left((q+1) \sum_{a\in \F_{q^2}^\times} (a\conj a,\p)^{kQ(\a)} -q \sum_{a\in \F_q^\times} (a\conj a,\p)^{kQ(\a)}\right) \\
&=-q^{2m-2}(q-1).
\end{align*}
 
Therefore the contribution to (\ref{cschi}) from such terms is (setting $m=(2l+1)n_\a$)
\[
\sum_{l=0}^\infty -q^{(4l+2)n_\a-2}(q-1) (q^{-2}x_\a)^{(2l+1)n_\a)} = \frac{-q^{-2}(q-1) x_\a^{n_\a} }{1-x_\a^{2n_\a}}.
\]
This covers all possible cases, and a simple addition of the rational functions we have obtained completes the proof.
\end{proof}

%%%%%%%%%%%%%%%%%%%%%%%%%%%%%%%%%%%%%%%%%%%%%%%%%%%%%%%%%%%%%%
\section{The action when $G_s\cong \Res SL_2$}%%%%%%%%%%%%%%%%
%%%%%%%%%%%%%%%%%%%%%%%%%%%%%%%%%%%%%%%%%%%%%%%%%%%%%%%%%%%%%%%

Let $s$ be a simple reflection in $W$.
Suppose $G_s\cong SL_2(E)$, $q$ is the cardinality of the residue field of $E$ and $\a$ is the unique positive coroot. The following theorem is equivalent to \cite[Lemma I.3.3]{kp}. We give a different and more direct proof, which will serve as a template for the more involved computation in the following section. Given any integer $t$,
define the Gauss sum $\mathfrak{g}_{SL_2(E)}(t)$ by
\[\mathfrak{g}_{SL_2(E)}(t)=
 \int_{O_E^\times} (v,\p)^t\psi\left(\frac{v}{\p}\right) dv
\] with a choice of Haar measure such that the total volume of $O_E^\times$ is $q-1$.

Recall the coefficients $\tau_{a,b}^{(s,\chi)}$ defined in (\ref{taudefn}). For two elements $\la,\mu\in X_*(S)$, we write $\la\sim\mu$ if $\la-\mu\in\La$.

\begin{theorem}\label{sl2tau}
Suppose that $a=\p^\nu$ and $b=\p^\mu$. Then we can write $\tau_{a,b}^{(s,\chi)}=\tau^1_{a,b}+\tau^2_{a,b}$ where
\begin{eqnarray*}
\tau^1_{a,b}&=&0\quad {\rm unless}\quad \nu \sim \mu \\%- \frac{B(\a,\mu)}{ Q(\a)}\a \\
\tau^2_{a,b}&=&0\quad {\rm unless}\quad s\nu \sim \mu-\a
\end{eqnarray*}
If $\nu=\mu$, then \[
\tau^1_{a,b}= (1-q^{-1})\frac{x_\a^{ n_\a \lceil \frac{B(\a,\mu)}{n_\a Q(\a)} \rceil }}{1-q^{-1}x_\a^{n_\a}}. \]
If $\nu=s\mu+\a$, then \[
\tau^2_{a,b}= q^{-1}\mathfrak{g}_{SL_2(E)}(B(\a,\mu)-Q(\a))\frac{1-x_\a^{n_\a}}{1-q^{-1}x_\a^{n_\a}}.\]
\end{theorem}
\begin{proof}
Write $n=(\begin{smallmatrix}
  1 & x \\
  0 & 1
\end{smallmatrix})\in G_s$ and $x=\p^{-m}v^{-1}$ where $v\in O_E^\times$ and $m\in \Z$.
The analogous statement to (\ref{su3fundm}) is $p(n)=(v,\p)^{mQ(\a)}v^\a \p^{m\a}u'$ for some $u'\in U$. We now calculate
\begin{align}
f_b(p(n)w_0)&=(v,\p)^{mQ(\a)} \delta^{1/2}(\p^{m\a}) \pi_\chi (v^\a \p^{m\a} )  f_b(w_0)\nonumber \\
&= q^{-m}(v,\p)^{mQ(\a)} \pi_\chi (b v^\a \p^{ma} )v_0\nonumber \\
&= q^{-m}(v,\p)^{mQ(\a)} [b,v^\a] \pi_\chi ( v^\a b \p^{m\a} )v_0\nonumber \\
&= q^{-m}(v,\p)^{mQ(\a)+B(\a,\mu)} \pi_\chi ( \p^{\mu+m\a} )v_0\label{fb}.
\end{align}
In the second line we use $f_b(w_0)=\pi_\chi(b)v_0$ and the $b$ ends up on the left of $v^\a \p^{m\a}$ since the $\pi_\chi$-action on functions in $I(\chi)$ is by $\pi_\chi(x)f(y)=f(yx)$.
In the last line we have used the commutator relation and the fact that $\phi_K(v^\a)=\phi_K(1)=v_0$ as $\phi_K$ is spherical.

Recall that we are trying to evaluate
$$
\tau_{a,b}^{(s,\chi)}= c_s(\chi)^{-1} \int_{U_s}\la_a^{(s\chi)} f_b( p(n) w_0 )  \psi^{-1}(n') dn
$$

We decompose $U_s$ into shells where $|x|$ is constant on each shell, and the above calculations show that
$$
\tau_{a,b}^{(s,\chi)}= c_s(\chi)^{-1}q^{-1} \la_a^{(s\chi)} \sum_{m\in \Z}
\pi_\chi ( \p^{\mu+m\a} ) \phi_K(1)
 \int_{O_E^\times} (v,\p)^{mQ(\a)+B(\a,\mu)} \psi(-\p^m v) dv.
$$ where a normalisation of Haar measure on $O_E^\times $ is chosen such that the group has volume $q-1$.

If $m\leq -2$, then the corresponding integral over $O_E^\times$ vanishes. Let us now consider contribution from the summand with $m=-1$. Here, the presence of the term $\la_a^{(s\chi)}\pi_\chi(\p^{\mu+m\a})v_0$
implies that this contribution is non-zero only when $s\nu\sim \mu-\a$. When $s\mu=\nu+\a$, the $m=-1$ term in the sum gives a contribution of $c_s(\chi)^{-1}q^{-1}\mathfrak{g}_{SL_2(E)}(B(\a,\mu)-Q(\a))$ to $\tau_{a,b}^{(s,\chi)}$. This corresponds to the $\tau_{a,b}^2$ part of the proposition.

Now we turn to the terms where $m\geq 0$. Here, the argument of $\psi$ is in $O_E$, so the character $\psi$ automatically takes the value 1. Hence, by Lemma \ref{integratecharacter}, the integral over $O_E^\times$ vanishes unless the integrand is identically equal to 1. This 
 occurs if and only if $B(\a,\mu)+mQ(\a)\equiv 0\pmod n$.

It was shown in the proof of \cite[Theorem 11.1]{mcn} that $B(\a,\mu)$ is divisible by $Q(\a)$. 
Thus the condition for non-vanishing of the integral now becomes $m\equiv -B(\a,\mu)/Q(\a) \pmod {n_\a}$. Let us write $m=kn_\a -B(\a,\mu)/Q(\a)$. Then $k$ runs over all integers greater than or equal to $\lceil \frac{B(\a,\mu)}{n_\a Q(\a)} \rceil$. By Lemma \ref{inh}, for integers $m$ in this family, all elements of the form $\p^{\mu+m\a}$ lie in the same $H$-coset. 

Therefore the sum over $m\geq 0$ is zero
unless $s\nu\sim \mu- \frac{B(\a,\mu)}{ Q(\a)}\a$. By the definition of the action of $W$ on the cocharacter lattice, this is equivalent to $\nu\sim \mu$.

If indeed we do have $\nu= \mu$, then we obtain a contribution to $\tau_{a,b}^{(s,\chi)}$ from the terms $m\geq 0$ equal to 
$$c_s(\chi)^{-1}(1-q^{-1})\sum_{k\geq \lceil \frac{B(\a,\mu)}{n_\a Q(\a)} \rceil} x_\a^{kn_\a}.$$ This corresponds to the $\tau_{a,b}^1$ part of the proposition and the only thing remaining to do is to sum this geometric series and substitute the evaluation of $c_s(\chi)$ from Theorem \ref{gkunramified}.
\end{proof}
\section{The action when $G_s\cong \Res SU_3$}\label{su3rankone}
%%%%%%%%%%%%%%%%%%%%%%%%%%%%%%%%%%%%%%%%%%%%%%%%%%%%%%%%%%%%%%

Now that we have warmed up by proving a Gindikin-Karpelevic formula and covered the computation of the coefficients $\tau_{a,b}^{(s,\chi)}$ when $G_s \cong \Res SL_2$, we present the analogue of Theorem \ref{sl2tau} when $G_s\cong \Res SU_3$.

So assume that $G_s\cong \Res_{E/F}SU_3$, let $q$ be the cardinality of the residue field of $E$ and let $\a$ be the unique positive rational coroot. Given any integer $t$, we define the $SU_3$ Gauss sum $\mathfrak{g}_{SU_3(E)}(t)$ by
\[
 \mathfrak{g}_{SU_3(E)}(t)=\sum_{u\in U(\F_q)\setminus \{1\}}(y\conj y,\p)^t \psi(\frac{x}{\p y})
\] where $u=\begin{pmatrix}
1  & x & y  \\
0 & 1 & -\conj{x}\\
0 & 0 & 1
\end{pmatrix}\in U(\F_q)$, the unipotent radical of $SU_3(\F_q)$.
%The author does not know any properties of this particular algebraic integer analogous to the $SL_2$ case.

\begin{theorem}\label{su3tau}
 Suppose that $c=\p^\nu$ and $b=\p^\mu$. Then we can write $\tau_{c,b}^{(s,\chi)}=\tau^1_{c,b}+\tau^2_{c,b}$ where
\begin{eqnarray*}
\tau^1_{c,b}&=&0\quad {\rm unless}\quad \nu \sim \mu \\
\tau^2_{c,b}&=&0\quad {\rm unless}\quad s\nu \sim \mu-2\a
\end{eqnarray*}
If $\nu=\mu$, then \[
\tau^1_{c,b}= \frac{(1-q^{-3})x_\a^{2n_\a\lceil \frac{B(\a,\mu)}{2n_\a Q(\a)}\rceil} +(q^{-1}-q^{-2})q^{(1-\epsilon)/2} x_\a^{ (2\lceil \frac{ B(\a,\mu) + n_\a Q(\a) - Q(\a) }{ 2n_\a Q(\a) } \rceil -1)n_\a }   }{ ( 1-\epsilon q^{-1} x_\a^{n_\a} ) ( 1+\epsilon q^{-2} x_\a^{n_\a} ) }
 \]
If $\nu=s\mu+2\a$, then \[
\tau^2_{c,b}=  q^{-3} \mathfrak{g}_{SU_3(E)}(B(\a,\mu)/2-Q(\a))\frac{ 1-x_\a^{2n_\a} }{ ( 1-\epsilon q^{-1} x_\a^{n_\a} ) ( 1+\epsilon q^{-2} x_\a^{n_\a} ) }.\]
\end{theorem}

\begin{proof}
We begin by collecting notation, mostly from previous sections, which we will use in this proof.

The unramified quadratic extension of $E$ used to define $SU_3$ is denoted $L$, we use a bar to denote the action of the nontrivial element of $\Gal(L/E)$ and $\theta\in O_L^\times$ is such that $\conj\theta=-\theta$.

The coroot $\a$ can be written as $\a_1+\a_2$ where $\a_1$ and $\a_2$ are simple roots for the geometric cocharacter lattice of $SU_3$ that form an orbit under the action of $\Gal(L/E)$.

An element in $U_s$ is written
$u=\begin{pmatrix}
1  & x & y  \\
0 & 1 & -\conj{x}\\
0 & 0 & 1
\end{pmatrix}$ where $x,y\in L$ are such that $x\conj x + y + \conj y = 0$.

For nonzero $u\in U_s$, we let $m=-v(y)$, $y=a\theta\p^{-m}$ and $z=x/y$. Then $1/y=-z\conj z/2+h\theta$ for some $h\in E$.

The element $p(u)$ is equal to $ (a,\conj a)^{mQ(\a)/2}a^{\a_1}\conj{a}^{\a_2} \p^{m\a}$ times an element of $U$.
The commutator of $a^{\a_1}\conj{a}^{\a_2}$ and $\p^\mu$ is $(a\conj a,\p)^{B(\a,\mu)/2}$. Thus the evaluation of the function $f_b$ analogous to (\ref{fb}) is
$$
f_b(p(u)w_0)=q^{-2m}(a\conj a,\p)^{(mQ(\a)+B(\a,\mu))/2} \pi_\chi(\p^{\mu+m\a}) v_0.
$$

After substitutions, the equation \eqref{taurankone} becomes
\begin{equation}\label{thesum}
\tau_{c,b}^{(s,\chi)}= \sum_{m\in \Z} \frac{q^{-2m}}{c_s(\chi)} \la_c^{(s\chi)}\big(  \pi_\chi(\p^{\mu+m\a}) v_0 \big)\int_{C_m} (a\conj a
,\p)^{(mQ(\a)+B(\a,\mu))/2}  \psi (z) du.
\end{equation}

First consider the terms in the sum where $m\leq -3$ and is odd. Parametrising $C_m$ by the coordinates $(z,h)$, the integrand factors into a piece $(a\conj{a},\p)^{(mQ(\a)+B(\a,\mu))/2}$ which only depends on $h$ and a piece $\psi(z)$ which only depends on $z$. Therefore the integral over $C_m$ contains a factor equal to $\int_{\p^{(m+1)/2}\mathcal{O}_L} \psi(z)dz$. 
This is zero by Lemma \ref{integratecharacter}.

Now suppose that $m\leq 3$ and is even. Pick $u\in O_L^\times$. Consider the measure preserving homeomorphism $(z,h)\mapsto(z+\p^{-1}u,h)$ of $C_m$. The effect of this homeomorphism on 
the corresponding term of (\ref{thesum}) is to multiply the
integrand by $\psi(u/\p)$. As $u$ can be chosen such that $\psi(u/\p)\neq 1$, these terms with $m\leq 3$ are all zero.

Consider the term where $m=-2$. It is
$$
c_s(\chi)^{-1}q^{4} \la_c^{(s\chi)} \big(\pi_\chi(\p^{\mu-2\a})v_0\big) \int_{C_{-2}} (a\conj a,\p)^{B(\a,\mu)/2-Q(\a)} \psi(x/y) du.
$$ 
The presence of the factor $\la_c^{(s\chi)} \pi_\chi(\p^{\mu-2\a})v_0$ implies that this term vanishes unless $s\nu\sim \mu-2\a$. When $\nu=s\mu+2\a$, by definition of the $SU_3$ Gauss sum, we get
$$
c_s(\chi)^{-1} q^{-3} \mathfrak{g}_{SU_3(E)}(B(\a,\mu)/2-Q(\a)).
$$
This corresponds to the $\tau_{a,b}^2$ term in the theorem.

Now suppose that $m\geq -1$. Then $z\in \mathcal{O}_L$ so $\psi(z)=1$.
Therefore we have to evaluate the integral
\begin{equation}\label{cmint}
\int_{C_m} (a\conj a,\p)^{(mQ(\a)+B(\a,\mu))/2}du.
\end{equation}

First consider the case where $m$ is odd.
Then $a\equiv \conj a \pmod{\p}$. The integral (\ref{cmint}) thus reduces to a constant factor times
\[
\sum_{a\in\mathbb{F}_q^\times} (a,\p)^{mQ(\a)+B(\a,\mu)}.
\]
We evaluate it by Lemma \ref{integratecharacter} to see that when $m$ is odd, (\ref{cmint}) evaluates to
\[
\int_{C_m} (a\conj a,\p)^{(mQ(\a)+B(\a,\mu))/2}du = 
\begin{cases}
                                         \operatorname{vol}(C_m) & \text{if $n\mid mQ(\a)+B(\a,\mu)$} \\
0 & \text{otherwise} \\ \end{cases}
\]

Now suppose that $m$ is even.
We note that $B(\a,\mu)\in 2\Z$ and $B(\a,\mu)/2$ is divisible by $Q(\a)$.
This is because $\a=\a_1+\a_2$ where $\a_1$ and $\a_2$ are in the same Galois orbit. Thus by linearity of $B$ and Galois-invariantness of $\mu$, we have $B(\a,\mu)=2B(\a_1,\mu)$, and we already know that $B(\a_1,\mu)$ is divisible by $Q(\a)$.
Let us write $m=2k$ and $B(\a,\mu)=2lQ(\a)$.

By the same volume computation that appeared in Section \ref{gkformula}, we have
\[
\int_{C_m} (a\conj a,\p)^{(mQ(\a)+B(\a,\mu))/2}du = q^{2m-3}\Big((q+1)\sum_{a\in\F_{q^2}^\times} (a\conj a,\p)^{(k+l)Q(\a)} - \sum_{a\in \F_q^\times } (a,\p)^{2(k+l)Q(\a)} \Big).
\]
Again, this can be evaluated using Lemma \ref{integratecharacter} to obtain, for $m$ even,
\[
\int_{C_m} (a\conj a,\p)^{(mQ(\a)+B(\a,\mu))/2}du = 
\begin{cases}
                                         \operatorname{vol}(C_m) & \text{if $n_\a\mid k+l$} \\
-q^{2m-2}(q-1) & \text{if $n_\a\nmid k+l$ and $n_\a\mid 2(k+l)$} \\
0 & \text{if $n_\a\nmid 2(k+l)$} \\ \end{cases}
\]

In all cases, we see that we get zero unless $n$ divides $mQ(\a)+B(\a,\mu)$. This is equivalent to $m\equiv -B(\a,\mu)/Q(\a)\pmod{n_\a}$. As $n_\a\a\in\La$, under such circumstances we have
$$
\mu+m\a\sim \mu-\frac{B(\a,\mu)}{Q(\a)}\a = s\mu.
$$

The factor $\la_c^{(s\chi)}(  \pi_\chi(\p^{\mu+m\a}) v_0)$ in (\ref{thesum}) is zero unless $s\nu\sim \mu+m\a$. Combining these observations shows that we only have a contribution to (\ref{thesum}) from terms with $m\geq -1$ when $\nu\sim \mu$.

To explicitly evaluate the term $\tau_{c,b}^1$ when $\nu=\mu$ is now a routine matter of evaluating the geometric series obtained as a consequence of the integral evaluations performed above.
\end{proof}

\section{Comparison}\label{comparison}
%%%%%%%%%%%%%%%%%%%%%%%%%%%%%%%%%%%%%%%%%

We conclude this paper by comparing our above results with those of Chinta and Gunnells \cite{cg} on the construction of the local part of a Weyl group multiple Dirichlet series. To do so, we suppose that $G$ is split, simple and simply connected, and that $Q$ is chosen to take the value 1 on short coroots.

The group homomorphism from $X_*(S)$ to $\G$, $\la\mapsto\sss(\p^\la)$ provides us with a splitting of the short exact sequence (\ref{ses}), and hence determines for us an isomorphism $\mathcal{O}(Y)\cong \C[\La]$. Let $\Sigma$ be the multiplicative subset of $\C[\La]$ such that $\mathcal{O}(V) \cong \Sigma\inv \C[\La]$.

Make a choice of left coset representatives $\Ga$ for $H$ in $\M$ consisting of elements of the form $\sss(\p^\la)$. The sections $\{\la_a^{(\chi)}\}_{a\in\Ga}$ of $\mathcal{E}^*$ provide us with a trivialisation of $\mathcal{E}^*$ and hence an isomorphism
\[
\Ga(V,\mathcal{E}^*)\cong \bigoplus_{a\in\Ga} \Sigma\inv \C[\La].
\]

Define an isomorphism $\varphi\map{\Ga(V,\mathcal{E}^*)}{\Sigma\inv \C[X_*(S)]}$ by
\begin{equation*}\label{phidefn}
\varphi\left( \sum_{\sss(\p^\mu)\in\Ga}  f_\mu\cdot \la_{\sss(\p^\mu)}^{(\chi)} \right) = \sum_{\sss(\p^\mu)\in\Ga} f_\mu \cdot \mu
\end{equation*}
where $f_\mu\in \mathcal{O}(V)$.

As the definition of the $W$-action on $\mathcal{E}^*|_V$ is by transporting the action of $W$ on $\mathcal{F}|_V$, we have, for $w\in W$,
\[
w\cdot \la_a^{(\chi)} = \sum_{b\in\Ga}\tau_{a,b}^{(w,w\inv\chi)} \la_b^{(w\inv\chi)}.
\]

At this point, let us introduce a change in notation. We shall write $\tau_{\mu,\nu}^i$ for $\tau_{\sss(\p^\mu),\sss(\p^\nu)}^i$ and $i=1,2$. For $\mu\in X_*(S)$, we now write $x^\mu$ for the corresponding element of the group algebra $\C[X_*(S)]$.

%If the coset representatives $\Ga$ are changed, the coefficients $\tau_{\mu,\nu}^{(w,\chi))}$ change by
%\[
%\tau_{\mu,\nu+\la}=x^\la \tau_{\mu,\nu},\quad \tau_{\mu+\la,\nu}=x^{-\la}\tau_{\mu,\nu}.
%\]

When $w=s$ is the simple reflection corresponding to the simple coroot $\a$ and $a=\sss(\p^\mu)$, Theorem \ref{sl2tau}
yields
\begin{equation}\label{firstsaction}
s\cdot  \la_a^{(\chi)} = \tau_{\mu,\mu}^1 \la_a^{(s\chi)} + \tau_{\mu,s\mu+\a}^2\la_{s\mu+\a}{(s\chi)}.
\end{equation}

Translating (\ref{firstsaction}) under the isomorphism $\varphi$ gives the following expression for the action of a simple reflection on $\Sigma\inv \C[X_*(S)]$:
\[
s\cdot x^\mu = \tau_{\mu,\mu}^1 x^\mu + \tau_{\mu,s\mu+\a}^2 x^{s\mu+\a}.
\]
Substituting in the values of $\tau^1$ and $\tau^2$ from Theorem \ref{sl2tau} yields the following explicit formula
\begin{multline*}
s_\a\cdot x^\la=\frac{x^{s\la}}{1-q^{-1}x_\a^{n_\a}}\Big( (1-q^{-1})x_\a^{n_\a \lceil \frac{B(\a,\la)}{n_\a Q(\a)} \rceil - \frac{B(\a,\la)}{Q(\a)} }  \\ +q^{-1}x_\a^{-1} (1-x_\a^{n_\a}) \mathfrak{g}_{SL_2(F)}(B(\a,\la)-Q(\a)) \Big)
\end{multline*}

We compare this to the action introduced in \cite{cg} for the purpose of constructing Weyl Group Multiple Dirichlet Series. Our lattice $X_*(S)$ corresponds to the lattice denoted $\La$ in \cite{cg}. As in
\cite[\S 9]{chintaoffen}, we need to make a minor change of variables from the Chinta-Gunnells paper to eliminate extraneous powers of $q$. Also as per \cite[\S 9]{chintaoffen}, we need to modify the action by multiplication with a rather innocuous factor in order to directly compare this action derived from intertwining operators with that from \cite{cg}.

In \cite[Definition 3.1]{cg}, an action is defined for any dominant $\la$. We write $f\mapsto f|_\la w$ for this action. Then $f||w:=x^{\la}(x^{-\la}f|_{\la}w)$ is independent of $\la$ and defines an action of $W$ on $\Sigma^{-1}X_*(S)$. 

Let $\Phi$ denote the set of coroots, and $\Phi=\Phi^+\sqcup \Phi^-$ the decomposition into positive and negative coroots. For $w\in W$, let $\Phi(w)=\{\a\in\Phi^+|w\a\in\Phi^-\}$.
Now consider the two Weyl group actions $\circ_1$ and $\circ_2$, under which the actions of $w\in W$ on $f\in \Sigma^{-1}X_*(S)$ are
\begin{align*}
w\circ_1 f&= \frac{c_{w_0}(\chi)}{c_{w_0}(w\chi)} (w\cdot f) \\
w\circ_2 f&=  \sgn(w) \left(\prod_{\substack \a\in\Phi(w)}x_\a^{n_\a}\right) (f||w).
\end{align*}
For any $w\in W$, $f\in \Sigma\inv \C[X_*(S)]$ and  $g\in\Sigma\inv \C[\La]$,
the actions $\cdot$ and $||$ satisfy the transformation property
\[
w\cdot (gf)= (wg) w\cdot f, \qquad (gf)||w = (wg)(f||w)
\]
where $wg$ denotes the usual action of $W$ on $\Sigma\inv \C[\La]$.
Therefore
 $\circ_1$ and $\circ_2$ are indeed Weyl group actions.
There is an explicit formula for $c_{w}(\chi)$, namely
\[
c_{w}(\chi)=\prod_{\a\in\Phi(w)} \frac{1-q^{-1}x_\a^{n_\a}}{1-x_\a^{n_\a}}
\]
where the product is over all positive coroots $\a$. This formula follows from Theorem \ref{gkunramified} and the cocycle formula $c_{w_1w_2}(\chi)=c_{w_1}(w_2\chi)c_{w_2}(\chi)$ whenever $\ell(w_1w_2)=\ell(w_1)+\ell(w_2)$.

\begin{proposition}
The two Weyl group actions $\circ_1$ and $\circ_2$ are the same action.
\end{proposition}

\begin{proof}
To prove this proposition, it suffices to consider the action of a simple reflection. We have explicit formulae for the action of a simple reflection on a monomial $x^\la$ in each case, and upon comparison, see that the actions are the same. The case of general $f$ then follows by linearity.
\end{proof}

Chinta and Gunnells use their action to construct the $p$-part of a Weyl group multiple Dirichlet series. This requires the construction of an auxiliary polynomial
$$
N(\chi,\la)=\prod_{\a > 0} \frac{1-q^{-1}x_\a^{n_\a}}{1-x_\a^{n_\a}} \sum_{w\in W} \sgn(w) \left(\prod_{\substack \a\in\Phi(w)}x_\a^{n_\a}\right) (1|_\la w)(\chi).
$$

Now the culmination of the above leads to our final result. Informally, this states that the value of the metaplectic Whittaker function on a torus element is equal to the $p$-part of a Weyl group multiple Dirichlet series.

\begin{theorem}
Let $\la$ be dominant. The following identity holds: $$(\delta^{-1/2}\W)(\p^{\la})=\chi(\p^\la) N(\chi,\la).$$ 
\end{theorem}

The results of this section suggest how to extend the results in \cite{cg} beyond the split case. We note that some of the content of this paper would accomplish some of the necessary work to achieve this aim. For example, we have obtained an independent proof of \cite[Theorem 3.2]{cg}. Although this proof is more indirect than that in \cite{cg}, it has the advantage of being more conceptual.

%%%%%%%%%%%%%%%%%%%%%%%%%%%%%%%%%%%%%%%%%%%%%%%
%\section{the spherical function}
%%%%%%%%%%%%%%%%%%%%%%%%%%%%%%%%%%%%%%%%%%%%%

%\bibliographystyle{alpha}
%\bibliography{bibfile}

\def\cprime{$'$}

\end{document}